\documentclass[envcountsect]{svjour3} 
\smartqed

\usepackage{graphicx,amssymb,amsmath,nicefrac,verbatim,colonequals,tikz,enumerate, mathtools,pgfplots}
\usepackage[sort&compress,numbers]{natbib}

\journalname{JOTA}

\newtheorem{observation}{\textbf{Observation}}

\newcommand{\R}{\mathbb{R}}

\newcommand{\otheta}{\varphi'}
\newcommand{\utheta}{\varphi}
\newcommand{\mtheta}{\bar{\varphi}_\gamma}
\begin{document}
	
	\title{Approximating Biobjective Minimization Problems Using General Ordering Cones}
	\titlerunning{Approx. Biobjective Minimization Problems Using General Ordering Cones}
	
	\author{Arne Herzel \and
		Stephan Helfrich \and 
		Stefan Ruzika			\and
		Clemens Thielen			
	}
	
	\institute{Arne Herzel (corresponding author) \and Stephan Helfrich \and Stefan Ruzika \at
		Department of Mathematics, University of Kaiserslautern,
		Paul-Ehrlich-Str.~14,\\ 67663~Kaiserslautern, Germany,
		\email{\{herzel,helfrich,ruzika\}@mathematik.uni-kl.de}
		\and
		Arne Herzel (corresponding author) \and Clemens Thielen \at 
		TUM Campus Straubing for Biotechnology and Sustainability, Weihenstephan-Triesdorf University of Applied Sciences, Am~Essigberg~3, 94315~Straubing, Germany, \email{\{arne.herzel,clemens.thielen\}@hswt.de} 
		\and 
		Clemens Thielen \at
		Department of Mathematics, Technical University of Munich, Boltzmannstr.~3,\\ 85748~Garching, Germany, \email{clemens.thielen@tum.de}
	}
	
	\date{Received: date / Accepted: date}
	
	\maketitle
	
	\begin{abstract}
		
		This article investigates the approximation quality achievable for biobjective minimization problems with respect to the Pareto cone by solutions that are (approximately) optimal with respect to larger ordering cones.
		When simultaneously considering $\alpha$-approximations for all closed convex ordering cones of a fixed inner angle~$\gamma \in [\frac \pi 2, \pi]$, an approximation guarantee between~$\alpha$ and~$2 \alpha$ is achieved, which depends continuously on~$\gamma$. The analysis is \mbox{best-possible} for any inner angle and it generalizes and unifies the known results that the set of supported solutions is a 2-approximation and that the efficient set itself is a 1-approximation. 
		
		Moreover, it is shown that, for maximization problems, no approximation guarantee is achievable by considering larger ordering cones in the described fashion, which again generalizes a known result about the set of supported solutions.
	\end{abstract}
	
	
	\section{Introduction}
	Multiobjective optimization problems, i.e., optimization problems with more than one objective function, are of growing interest in both mathematical optimization theory and real-world applications. In these problems, solutions optimizing all objectives simultaneously usually do not exist. Therefore, if no prior information about preferences is available, every  so-called \emph{efficient solution} is a possible candidate for an optimal solution. A solution is said to be efficient if any other solution that is better in some objective is necessarily worse in at least one other objective. One of the major challenges in multiobjective optimization is the overwhelming number of different images of efficient solutions that typically exist.
	
	Additional preference information reduces the number of solutions that qualify as optimal. A common way to model such preferences is via \emph{ordering cones}, which describe, for each solution, which other solutions are guaranteed to be worse. In the case of minimization problems, the case of no prior information described above corresponds to the ordering cone being the nonnegative orthant of the objective space (also called the \emph{Pareto cone} in this context). A larger ordering cone means more preference information and, thus, a smaller set of possible optimal solutions. Prominent special cases are weighted sum scalarizations, which correspond to the ordering cones being half spaces. If the weights for a weighted sum scalarization are given, this means that the complete preference information is available.
	
	Another important approach for dealing with large numbers of required solutions is the concept of approximation, where every solution only has to be covered up to a multiplicative tolerance in each objective function, thus reducing the number of needed solutions drastically.
	
	In this article, we study relations between these two approaches. More precisely, we study approximation properties (with respect to the Pareto cone) of solutions that are \textcolor{black}{(approximately)} optimal with respect to larger ordering cones. Our main focus lies on the case of biobjective minimization problems.
	\subsection{Related Work}
	\textcolor{black}{The field of study of mathematical optimization with respect to vector-valued objective functions and general preference relations is known as vector optimization. An introduction to the concepts of vector optimization can be found in~\cite{Eichfelder+Jahn:solution-concepts, jahn2009vector}. Multiobjective optimization is as a subfield of vector optimization in which preferences are defined by the componentwise ordering on~$\R^p$.}
	
	The use of cones to model preferences is a well-studied topic in multiobjective optimization~\cite{Ehrgott:book,Hunt+etal:cones, Yu:book} and their investigation as dominance cones was initiated by Yu~\cite{Yu:cones}. He gives an in-depth study of the equivalence of properties between orderings and cones in multiobjective and vector optimization theory. Conditions under which multiobjective optimization problems using alternative ordering cones can be reduced to the standard case of the componentwise ordering are studied in~\cite{Noghin:cones}. An overview about results on properties of ordering cones in multiobjective optimization and the corresponding literature can be found in~\cite{Wiecek:cone-advances}.
	
	Vanderpooten et al.~\cite{Vanderpooten+etal:covers+approximations} introduce a general framework modeling a variety of notions of approximation in the context of general ordering cones, including the concepts considered here. They provide conditions under which an approximation with respect to some cone is an approximation with respect to some other cone containing it. 
	Engau and Wiecek~\cite{Engau+Wiecek:cone-approximation} characterize an additive notion of approximation using the theory of dominance cones.
	
	The systematic study of the theory of approximation in multiobjective optimization in the multiplicative sense considered here started with the seminal work of Papadimitriou and Yannakakis~\cite{Papadimitriou+Yannakakis:multicrit-approx}. They show that, under weak assumptions, approximations of polynomial cardinality are guaranteed to exist and that the problem of finding an approximation can be polynomially reduced to solving an approximate version of the decision problem associated with the multiobjective optimization problem. Subsequent articles focus on sufficient conditions for the computability of approximations and their cardinality~\cite{Bazgan+etal:min-pareto,Diakonikolas+Yannakakis:epsilon-convex,Diakonikolas+Yannakakis:approx-pareto-sets, Herzel+etal.:one-exact,Koltun+Papadimitriou:approx-dom-repr,Vassilvitskii+Yannakakis:trade-off-curves}. A survey on literature about approximation methods for general multiobjective optimization problems and for several specific multiobjective combinatorial optimization problems is given in~\cite{Herzel+etal:survey}.
	
	The weighted sum scalarization (see, e.g.,~\cite{Ehrgott:book}) as a special case of alternative ordering cones has been a widely studied tool for computing approximations in multiobjective optimization problems: Gla{\ss}er et al.~\cite{Glasser+etal:multi-hardness} study how multiobjective optimization problems can be approximated using a norm-based approach. Most notably, they show that, for $p$-objective minimization problems, for any $\varepsilon > 0$, a $(p+\varepsilon)$-approximation can be computed using the weighted sum scalarization. A specific algorithm using the weighted sum scalarization for computing approximations in biobjective minimization problems is given in~\cite{Halfmann+etal:general-approx}. For biobjective optimization problems with convex feasible sets and linear objective functions, an efficient algorithm for computing $(1+\varepsilon)$-approximations is studied in~\cite{Daskalakis+etal:Chord-Algorithm}. For an extensive study of the approximation quality achievable by the weighted sum scalarization for multiobjective minimization and maximization problems in general, see~\cite{Bazgan+etal.:power-weighted-sum}.
	
	\subsection{Our Contribution}
	We consider multiplicative approximation using general ordering cones for the special case of biobjective minimization problems. More specifically, we investigate how optimal (or approximate) solutions with respect to general ordering cones can be used to achieve an approximation guarantee with respect to the usual Pareto cone. In contrast to the results by Vanderpooten et al.~\cite{Vanderpooten+etal:covers+approximations} about approximation guarantees carrying over from smaller to larger ordering cones, we show that an approximation with respect to some fixed ordering cone containing the Pareto cone does not straightforwardly yield an approximation with respect to the Pareto cone (i.e., in the classical sense). We introduce the concept of \emph{$\gamma$-supportedness} as a generalization of both supportedness and efficiency. For some angle~$\gamma \in [\frac \pi 2, \pi]$, a solution is called $\gamma$-supported if it is optimal with respect to some (arbitrary) ordering cone of inner angle $\gamma$. Thus, the definitions of a $\frac \pi 2$-supported solution and a $\pi$-supported solution coincide with the definition of an efficient solution and a supported solution, respectively. We show that this characterization of ordering cones by their inner angle provides structural results on the approximation guarantee that is achievable for the Pareto cone by solutions that are approximately optimal with respect to larger cones. Our main result (Theorem~\ref{thm:cone-approx}) naturally generalizes existing approximation results for the weighted sum scalarization as well as for the Pareto cone and unifies them in a general statement about approximability by a family of cones specified by their inner angle. Moreover, we show that the achieved approximation guarantees are best possible for every inner angle $\gamma \in [\frac \pi 2, \pi]$, including the previously known cases. Finally, we show that considering families of cones of the same inner angle does not yield an approximation guarantee for maximization problems, which, again, generalizes known results for the weighted sum scalariztion to general ordering cones.
	
	\section{Preliminaries}\label{sec:preliminaries}
	In this section, we first repeat some important concepts and definitions from multiobjective optimization theory in the classical sense. Then we briefly recall how to generalize multiobjective optimization problems to more general ordering relations via cones and provide some basic properties of this generalization. 
	
	We introduce a new framework that allows us to describe biobjective minimization problems with respect to general ordering relations and to define $\gamma$-supportedness. Finally, we provide a formal definition of approximation for multiobjective optimization problems with respect to general ordering cones. 
	\subsection{Multiobjective Optimization and Scalarizations}\label{subsec:moo}
	We use the usual notation $\R^p_\geqq\coloneqq \{y \in \R^p : 0 \leqq y \}$, where $0 \in \R^p$ is the $p$-dimensional zero vector and $\leqq$ is the weak componentwise order:
	\begin{align*}
	y \leqq y' &\ratio\Leftrightarrow y_i \leq y'_i, \quad i = 1, \ldots,p
	\end{align*}
	Multiobjective optimization problems can be formally defined as follows:
	\begin{definition} [Multiobjective Minimization/Maximization Problem]
		For $p \geq 1$, a \emph{$p$-objective optimization problem}~$\Pi$ is given by a set of instances. Each instance $I=(X^I,f^I)$ consists of a (finite or infinite) set~$X^I$ of (feasible) solutions and a vector~$f^I = (f^I_1,\ldots, f^I_{p})$ of $p$ objective functions~$f^I_i : X^I \to \R$ for $i = 1,\ldots,p$. In a minimization problem, all objective functions~$f^I_i$ should be minimized, in a maximization problem, they should be maximized.
	\end{definition}
	
	The solutions of interest are those for which it is not possible to improve the value of one objective function without worsening the value of at least one other objective. Solutions with this property are called \emph{efficient solutions}:
	
	\begin{definition}
		For an instance~$I=(X^I,f^I)$ of a $p$-objective minimization (maximization) problem, a solution~$x \in X^I$ \emph{dominates} another solution~$x' \in X^I$ if $f^I(x) \neq f^I(x')$ and $f^I(x) \leqq f^I(x')$ ($f^I(x) \geqq f^I(x')$). A solution~$x \in X^I$ is called \emph{efficient} if it is not dominated by any other solution~$x' \in X^I$. The set~$X^I_E\subseteq X^I$ of all efficient solutions is called the \emph{efficient set}.
	\end{definition}
	
	In the following, we usually drop the superscript~$I$ indicating the dependence on the instance in~$X^I$, $f^I$, etc. The majority of the results of this paper are only applicable for minimization problems. Therefore, we introduce some of the concepts in this chapter for minimization problems only, even though they easily transfer to the case of maximization. Some of the formal definitions for maximization problems are given in Section~\ref{sec:maximization}.
	
	In the remainder of this paper, it is assumed that, in any instance~$I = (X,f)$ of a $p$-objective minimization problem, the set $f(X) + \R^p_\geqq$ is closed.\footnote{This property is known as $\R^p_\geqq$-closedness~\cite{Ehrgott:book}.} Note that this is, in particular, the case if $f(X)$ is \textcolor{black}{compact}, which holds, for example, if $f(X)$ is finite or a \textcolor{black}{polytope}. Additionally, it is assumed that all objective functions only attain positive values $f_i(x) > 0$ for all $x \in X$ and $i = 1,\ldots, p$. This allows for a reasonable notion of approximation (see Subsection~\ref{subsec:approx}). These assumptions imply that, for any feasible solution~$x \in X$ that is dominated by another feasible solution $x' \in X$, there also exists an efficient solution~$x'' \in X_E$ dominating~$x$.\footnote{This property is known as external stability~\cite{Ehrgott:book}.}
	
	\medskip
	
	When dealing with multiobjective optimization problems, it is common to consider \emph{scalarizations}, where related single objective optimization problems are considered in order to gain information about the multiobjective problem at hand. Here, we consider only scalarizations where the feasible set remains unchanged. We call an instance of a single objective optimization problem that shares the feasible set~$X$ with a given multiobjective optimization problem instance~$I$ (and whose solutions yield some information about the multiobjective instance) a \emph{scalarization of~$I$}.
	
	Two of the most important kinds of scalarizations are weighted sum scalarizations and weighted max-ordering scalarizations. 
	
	\begin{definition}
		For an instance~$I = (X,f)$ of a $p$-objective minimization problem and weights~$w_i > 0$ for $i = 1, \ldots, p$, the \emph{weigthed sum scalarization of~$I$ with weights~$w_1,\ldots,w_p$} is the single objective instance
		\begin{align*}
		\min_{x \in X} \quad w_1 \cdot f_1(x) + \cdots + w_p \cdot f_p(x).
		\end{align*}
	\end{definition}
	
	It is well-known that, for any multiobjective optimization problem instance~$I$ and weights~$w_i > 0$ for $i=1,\ldots,p$, any solution~$x \in X$ that is optimal for the weighted sum scalarization of~$I$ with weights~$w_1,\ldots,w_p$ is efficient (for~$I$). On the other hand, there might exist efficient solutions that are not optimal for any weighted sum scalarization. Solutions that are optimal for some weighted sum scalarization are called \emph{supported solutions}.
	
	\begin{definition}
		For an instance~$I = (X,f)$ of a $p$-objective minimization problem and weights~$w_i > 0$ for $i = 1, \ldots, p$, the \emph{weighted max-ordering scalarization of~$I$ with weights~$w_1,\ldots,w_p$} is the single objective instance
		\begin{align*}
		\min_{x \in X} \quad \max\left\{ w_1 \cdot f_1(x), \ldots, w_p \cdot f_p(x)\right\}.
		\end{align*}
	\end{definition}
	
	It is well-known that, for any multiobjective optimization problem instance~$I$ and weights~$w_i > 0$, there exists some solution~$x \in X$ that is optimal for the weighted max-ordering scalarization of~$I$ with weights~$w_1,\ldots,w_p$ and also efficient (for~$I$). Moreover (if $f(x) > 0$ for all $x \in X$ as assumed here), each efficient solution~$x \in X_E$ is optimal for the weighted max-ordering scalarization with weights~$w_i = \frac 1 {f_i(x)}$ for $i = 1,\ldots,p$.

	\subsection{Orderings and Cones}\label{subsec:cones}
	In multiobjective minimization problems, where efficient solutions are of interest, it is implicitly assumed that the underlying preference relation is the weak componentwise order $\leqq$: A solution $x \in X$ is efficient if and only if, for any $x' \in X$ with $f(x') \leqq f(x)$, we also have $f(x) \leqq f(x')$. However, this can be generalized to other reasonable ways of defining ``optimal'' solutions.
	
	\textcolor{black}{A binary relation~$R$ on a vector space that is reflexive, transitive, compatible with addition (i.e., for any~$y, y',z \in \R^p$ with $y R y'$, we have $(y+z) R (y'+z)$), and compatible with scalar multiplication (i.e., for any~$y,y' \in \R^p$ with $yRy'$ and any~$\lambda > 0$, we have $(\lambda \cdot y) R (\lambda\cdot y')$) is called a \emph{vector preorder}. It is well-known that any closed vector preorder~$R$ on $\R^p$ corresponds to exactly one closed convex cone~$C \subseteq \R^p$ via $y R y' \Leftrightarrow y' - y \in C$ and vice versa~\cite{Ehrgott:book}.}
	
	In multiobjective optimization, the relations that are of interest additionally adhere to \textcolor{black}{the so-called \emph{Pareto axiom}~\cite{Noghin:cones}:} If a solution is at least as good as another solution in all objective functions, it should also be at least as good in the multiobjective sense, and if a solution is not better than another solution in any objective function and strictly worse in at least one objective, it should be worse in the multiobjective sense. For multiobjective minimization problems, this means that a closed vector preorder~$\preceq$ only qualifies as a meaningful way to describe multiobjective preferences if we have~$\R^p_\geqq \subseteq C_\preceq$ and $- \R^p_\geqq \cap C_\preceq = \{0\}$. 
	
	In the two-dimensional case, the situation is particularly simple: Any closed convex cone $C \subseteq \R^2$ (except for the empty set and subspaces of $\R^2$) can be uniquely described by its inner angle $\gamma \in [0, \pi]$ and its rotation $\utheta \in [0, 2\pi)$ with respect to some direction of reference. For cones containing $\R^2_\geqq$, the inner angle~$\gamma$ has to be within $[\frac \pi 2, \pi]$ and the angle of rotation~$\utheta$ can vary within an interval of length~$\gamma - \frac \pi 2$ (without loss of generality, the interval $[0, \gamma - \frac \pi 2]$ since we can choose the direction of reference accordingly). Note that, if the inner angle of a cone containing~$\R^2_\geqq$ is smaller than $\pi$, it does not contain any point from~$-\R^2_\geqq \setminus \{0\}$. There exist exactly two cones of inner angle $\pi$ that contain~$\R^2_\geqq$ and are not disjoint from~$-\R^2_\geqq \setminus \{0\}$, namely $\{(y_1,y_2) \in \R^2 | y_1 \geq 0\}$ and $\{(y_1,y_2) \in \R^2 | y_2 \geq 0\}$. Thus, in a closed convex cone~$C \subseteq \R^2$, if the inner angle~$\gamma$ is smaller than $\pi$, we have~$\R^p_\geqq \subseteq C$ and $- \R^p_\geqq \cap C = \{0\}$ if and only if the angle of rotation $\utheta$ lies in the closed interval~$[0, \gamma - \frac \pi 2]$. If the inner angle~$\gamma$ is equal to $\pi$, we have~$\R^p_\geqq \subseteq C$ and $- \R^p_\geqq \cap C = \{0\}$ if and only if $\utheta$ lies in the open interval~$(0, \frac \pi 2)$. Given $\gamma \in [\frac \pi 2, \pi]$, we can write the allowed interval for $\utheta$ shortly as $[0,\gamma -\frac \pi 2] \setminus \{\gamma - \pi, \frac\pi 2\}$. This yields the closed interval~$[0, \gamma- \frac \pi 2]$ for $\gamma < \pi$ and the open interval~$(0, \gamma - \frac \pi 2) = (0, \frac \pi 2)$ for $\gamma = \pi$.
	
	Hence, the following definition covers exactly all closed convex cones~$C \subseteq \R^2$ for which~$\R^p_\geqq \subseteq C$ and $- \R^p_\geqq \cap C = \{0\}$.
	
	\begin{definition}
		\textcolor{black}{For $\gamma \in [\frac \pi 2, \pi]$ and $\utheta \in [0, \gamma- \frac \pi 2]\setminus \{\gamma - \pi, \frac \pi 2\}$, we define
			\begin{align*}
			\otheta \coloneqq \gamma - \frac \pi 2 - \utheta.
			\end{align*}
			In the following, if the values of~$\gamma$ and~$\utheta$ are clear from the context, we always use this convention. We define a linear mapping $T_\gamma^\utheta : \R^2 \to \R^2$ via
			\begin{align*}
			T_\gamma^\utheta(y) \coloneqq \left(
			\begin{array}{cc}
			\sin \gamma & (-\cos \gamma)\\
			0&1
			\end{array}
			\right)
			\left(
			\begin{array}{cc}
			\cos \utheta & (-\sin \utheta)\\
			\sin \utheta&\cos \utheta
			\end{array}
			\right) \cdot y = 
			\left(
			\begin{array}{cc}
			\cos \otheta & \sin \otheta\\
			\sin \utheta&\cos \utheta
			\end{array}
			\right) \cdot y.
			\end{align*}
			Using this notation, we define a cone
			\begin{align*}
			C_\gamma^\utheta \coloneqq \{y \in \R^2 : T_\gamma^\utheta(y) \geqq 0 \}
			\end{align*}
			and the corresponding vector preorder~$\leqq_\gamma^\utheta$ on $\R^p$ by
			\begin{equation*}
			y \leqq_\gamma^\utheta y' \quad \ratio \Longleftrightarrow \quad  y' - y \in C_\gamma^\utheta.
			\end{equation*} 	 
		}
		For~$\gamma \in [\frac \pi 2, \pi]$, we define $\mtheta$ to be the value of $\utheta \in [0, \gamma- \frac \pi 2]\setminus \{\gamma - \pi, \frac \pi 2\}$ for which $\otheta = \utheta$:
		\begin{align*}
		\mtheta \coloneqq \frac \gamma 2 - \frac \pi 4
		\end{align*}
	\end{definition}
	
	\textcolor{black}{It is easy to see that~$C_\gamma^\utheta \subsetneq \R^2$ is a closed convex cone with inner angle~$\gamma$ containing~$\R^2_\geqq$, and that the extreme directions of $C_{\gamma}^{\utheta}$ include angles of~$\utheta$ and~$\gamma - \frac \pi 2 - \utheta$ with the first axis and second axis, respectively (see Figure~\ref{fig:coneintroduction}): The first $2 \times 2$-matrix in the definition of~$T_\gamma^\utheta$ rotates the first axis by an angle of~$\gamma$ while the second axis remains unchanged. The second $2 \times 2$-matrix is a rotation matrix with rotation angle $\utheta$.}
	
	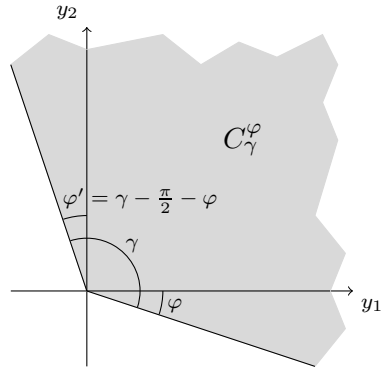
\begin{figure}[ht!]
		\begin{center}
			\begin{tikzpicture}[scale=1]
			\draw[->] (1,2) -- (5.5,2) node[below right] {$y_1$};
			\draw[->] (2,1) -- (2,5.5) node[above left] {$y_2$};

			\draw[fill,gray!30] (5,1) -- (2,2) -- (1,5)--(1.5,5.4) -- (2,5.2) -- (5.2,2) -- (5.4,1.5) -- (5,1);
			\draw[fill,gray!30] (2,2) -- (2,5.2) --(3,5.4) -- (3.5,5)-- (4,5.3) -- (4.5,5.1) -- (5.1,5.4) -- (5.3,5) -- (5.1,4.5) -- (5.3,4) -- (5,3)--(5.4,2.5) --  (5.2,2) -- (2,2);

			\draw[-] (2,2) -- (5.2,2);
			\draw[-] (2,2) -- (2,5.2);
			
			\draw[-] (2,2) -- (5,1);
			\draw[-] (2,2) -- (1,5);

			\node[] at (4,4)  {{ \large $C_\gamma^\utheta$}};
			
			
			\draw (3,2) arc (0:-18:1);
			\draw (2,3) arc (90:108:1);
			\draw (2.7,2) arc (0:-18:0.7);
			\draw (2.7,2) arc (0:108:0.7);
			
			\node[] at (3.16,1.81)  {$\utheta$};
			\node[] at (2.6,2.6)  {$\gamma$};
			\node[] at (2.7,3.2)  {$\otheta = \gamma - \frac \pi 2 - \utheta$};
			\end{tikzpicture}
			\caption{Illustration of the cone~$C_\gamma^\utheta \subsetneq \R^2$.\label{fig:coneintroduction}}
		\end{center}
	\end{figure}
	\noindent \textcolor{black}{Moreover, the following lemma holds for $\leqq_\gamma^\utheta$:}
	{\color{black}
		\begin{lemma}\label{lem:order-transformation}
			For $y,y' \in \R^2$, we have $y \leqq _\gamma^\utheta y'$ if and only if $T_\gamma^\utheta(y) \leqq T_\gamma^\utheta(y')$.
		\end{lemma}
		\begin{proof}
			We have
			\begin{align*}
			y \leqq_\gamma^\utheta y' \quad\Leftrightarrow\quad y' - y \in C_\gamma^\utheta \quad\Leftrightarrow\quad T_\gamma^\utheta(y' - y) \geqq 0 \quad\Leftrightarrow\quad T_\gamma^\utheta(y') \geqq T_\gamma^\utheta(y)
			\end{align*}
			by the definitions of~$\leqq_\gamma^\utheta$ and $C_\gamma^\utheta$ and by linearity of $T_\gamma^\utheta$.
			\qed
		\end{proof}
	}
	
	\noindent We summarize the facts obtained in this subsection so far in the following proposition:
	
	\begin{proposition} \label{prop:all-orders-covered}
		Let $C \subseteq \R^2$. The following statements are equivalent:
		\begin{enumerate}
			\item $C = C_\gamma^\utheta$ for some $\gamma \in [\frac \pi 2, \pi]$ and $\utheta \in  [0, \gamma- \frac \pi 2]\setminus \{\gamma - \pi, \frac \pi 2\}$.
			\item $C$ is a closed convex cone with $\R^2_\geqq \subseteq C$ and $- \R^p_\geqq \cap C = \{0\}$.
			\item $C = C_\preceq$ for a closed vector preorder~$\preceq$ on $\R^2$ for which $y \leqq y'$ implies $y \preceq y'$, and $y \leqq y'$ and $y \neq y'$ imply $y' \npreceq y$ for all $y,y' \in \R^2$.
		\end{enumerate}
	\end{proposition}
	
	Given some $\gamma \in [\frac \pi 2, \pi]$, $\utheta \in  [0, \gamma- \frac \pi 2]\setminus \{\gamma - \pi, \frac \pi 2\}$, and an instance $(X,f)$ of a biobjective minimization problem, we can define a biobjective minimization problem instance with the same feasible set~$X$ and objective function~$f$, but using~$\leqq_\gamma^\utheta$ instead of~$\leqq$ as the underlying vector preorder. Proposition~\ref{prop:all-orders-covered} states that  any reasonable way to define minimization of~$f$ over~$X$ can be described like this. Moreover, from \textcolor{black}{Lemma~\ref{lem:order-transformation}}, we know that, for any biobjective minimization problem instance~$(X,f)$, using~$\leqq_\gamma^\utheta$ is equivalent to using the weak componentwise order~$\leqq$ for the objective function~$T_\gamma^\utheta \circ f: X \rightarrow \R^2$.
	\begin{definition}\label{def:instance}
		For $\gamma \in [\frac \pi 2, \pi]$, $\utheta \in  [0, \gamma- \frac \pi 2]\setminus \{\gamma - \pi, \frac \pi 2\}$, and an instance~$I = (X,f)$ of a biobjective minimization problem~$\Pi$,  we define~$I_\gamma^\utheta \coloneqq (X, T_\gamma^\utheta \circ f)$:
		\begin{align*}
		\min_{x\in X} \quad T_\gamma^\utheta \left(f(x)
		\right)
		\end{align*}
		In a biobjective minimization problem instance~$I = (X,f)$, we say that a solution~$x \in X$ is \emph{optimal with respect to $\leqq_\gamma^\utheta$} if $x$ is efficient in~$I_\gamma^\utheta$, i.e., if there does not exist a solution $x' \in X$ such that $f(x') \neq f(x)$ and $f(x') \leqq_\gamma^\utheta f(x)$.\footnote{Note that, for any $\gamma \in [\frac \pi 2, \pi]$ and $\utheta \in  [0, \gamma- \frac \pi 2]\setminus \{\gamma - \pi, \frac \pi 2\}$, if $f(x) > 0$, then also $T_\gamma^\utheta(f(x)) > 0$. Thus, $I_\gamma^\utheta$ indeed always satisfies our assumption of positive-valued objective functions. Moreover, this implies that our assumption of $f(X) + \R^2_\geqq$ being closed also transfers to $I_\gamma^\utheta$.}
	\end{definition}
	
	The above reasoning implies that solving a biobjective minimization problem instance with respect to any reasonable closed vector preorder can be reduced to applying a linear mapping and solving the resulting instance with respect to the usual componentwise order. Thus, for any $\gamma \in [\frac \pi 2, \pi]$ and $\utheta \in  [0, \gamma- \frac \pi 2]\setminus \{\gamma - \pi, \frac \pi 2\}$, any known result that holds for biobjective optimization problems in the usual sense can also be applied to $I_\gamma^\utheta$ as long as all of the corresponding conditions are satisfied.
	However, one has to be careful when applying algorithmic results to $I_\gamma^\utheta$ since basic requirements like, e.g., polynomial computability of the objective function, do not trivially hold for $(T_\gamma^\utheta \circ f)$ even if $f$ is polynomially computable as the matrix describing~$T_\gamma^\utheta$ might contain irrational entries. 
	
	Obviously, $T_{\frac \pi 2}^0$ is the identity mapping, so, for any instance~$I$ of a biobjective minimization problem, we have $I_{\frac \pi 2}^0 = I$. Thus, in the special case~$\gamma = \frac \pi 2$ and (thus)~$\utheta = \otheta = 0$, the optimal solutions with respect to~$\leqq_\gamma^\utheta$ are exactly the efficient solutions. In the other extreme case, where $\gamma = \pi$ and $\utheta \in [0, \gamma- \frac \pi 2]\setminus \{\gamma - \pi, \frac \pi 2\} = (0, \frac \pi 2)$, the definition of $I_\gamma^\utheta$ yields the single objective optimization problem instance
	\begin{align*}
	\min_{x\in X} \quad \sin \utheta \cdot f_1(x) + \cos \utheta \cdot f_2(x),
	\end{align*}
	i.e., the weighted-sum scalarization of $I$ with (positive) weights~$\sin \utheta$ and~$\cos \utheta$.
	
	Recall that, in a multiobjective optimization problem, a solution~$x \in X$ is called supported if there exists a nonnegative vector of weights such that $x$~is an optimal solution of the weighted sum scalarization with these weights. Equivalently, using the fact that weighted sum scalarizations correspond to the case of the inner angle $\gamma$ being equal to $\pi$, we can say that a solution is supported if and only if it is an optimal solution of~$I_\gamma^\utheta$ for~$\gamma = \pi$ for some $\utheta \in (0,\frac \pi 2)$. We generalize this idea to arbitrary values of $\gamma \in [\frac \pi 2, \pi]$ in the following way:
	\begin{definition}\label{def:gamma-supp}
		Let $I = (X,f)$ be a biobjective optimization problem and let $\gamma \in [\frac \pi 2, \pi]$ be given. We say that a solution $x \in X$~is \emph{$\gamma$-supported} if there exists some~$\utheta \in [0, \gamma - \frac \pi 2]\setminus \{\gamma - \pi, \frac \pi 2\}$ such that $x$~is optimal with respect to $\leqq_\gamma^\utheta$.
	\end{definition}
	Hence, the definition of a supported solution coincides with the definition of a $\pi$-supported solution. Moreover, the definition of an efficient solution is exactly the definition of a $\frac \pi 2$-supported solution. Thus, the concept of $\gamma$-supportedness generalizes and connects the concepts of efficiency and supportedness. \textcolor{black}{Note that, if~$\gamma_1,\gamma_2 \in [\frac \pi 2 , \pi]$ such that $\gamma_1 \leq \gamma_2$, then every $\gamma_2$-supported solution is $\gamma_1$-supported. In particular, for any $\gamma \in [\frac \pi 2, \pi]$, every supported solution is $\gamma$-supported and every $\gamma$-supported solution is efficient. }
	\subsection{Approximation}\label{subsec:approx}
	Next, we define approximation for biobjective minimization problems (the definition for maximization problems is analogous). Here, we generalize the usual notion of approximation, which is based on the componentwise order, to arbitrary ordering relations on~$\R^2$. The usual definition of approximation (see~\cite{Herzel+etal:survey}) is obtained by replacing $\leqq_\gamma^\utheta$ by $\leqq$ in the following definition.\footnote{In the case of $\leqq$, an approximation is often called an ``approximate Pareto set''. Here, we only use the term ``approximation'', which also fits the more general case.}
	
	\begin{definition}\label{def:approximation}
		Let~$I = (X,f)$ be a biobjective minimization problem instance such that $f_1(x),f_2(x) > 0$ for all $x \in X$. Let $\gamma \in [\frac \pi 2, \pi]$ and $\utheta \in  [0, \gamma- \frac \pi 2]\setminus \{\gamma - \pi, \frac \pi 2\}$ be given. For a scalar~$\alpha \geq 1$, we say that \emph{$x' \in X$ is $\alpha$-approximated by~$x \in X$ with respect to~$\leqq_\gamma^\utheta$} if $f(x) \leqq_\gamma^\utheta \alpha \cdot f(x')$. A set $X_\alpha \subseteq X$ is called an \emph{$\alpha$-approximation with respect to $\leqq_\gamma^\utheta$} if any feasible solution~$x \in X$ is $\alpha$-approximated with respect to $\leqq_\gamma^\utheta$ by some solution $x' \in X_\alpha$.
		
		In single objective minimization problem instances, we say that a solution is \emph{$\alpha$-approximate} if it $\alpha$-approximates any other feasible solution (in the single objective sense, where $x'$ is $\alpha$-approximated by~$x$ if $f(x) \leq \alpha \cdot f(x)$).
	\end{definition}
	
	Obviously, for any biobjective minimization problem instance, the efficient set is a $1$-approximation.
	
	Definition~\ref{def:approximation}, together with \textcolor{black}{Lemma~\ref{lem:order-transformation}}, states that a solution $x' \in X$ is $\alpha$-approximated by another solution~$x \in X$ with respect to $\leqq_\gamma^\utheta$ if $T_\gamma^\utheta(f(x)) \leqq T_\gamma^\utheta(\alpha \cdot f({\color{black}x'}))$.
	Note that, by linearity of $T_\gamma^\utheta$, this is equivalent to $T_\gamma^\utheta(f(x)) \leqq \alpha \cdot T_\gamma^\utheta(f({\color{black}x'}))$. Thus, $x' \in X$ is $\alpha$-approximated by~$x \in X$ with respect to $\leqq_\gamma^\utheta$ in $I$ if and only if $x'$ is $\alpha$-approximated by $x$ (with respect to $\leqq$) in $I_\gamma^\utheta$. Recall that, in the biobjective case, optimization with respect to any closed vector preorder can be reduced to the componentwise order via $T_\gamma^\utheta$. The above reasoning states that the concept of approximation is consistent with this reduction. In fact, this equivalent characterization of approximation would be a different straightforward way to define approximation with respect to~$\leqq_\gamma^\utheta$. However, the definition as stated in Definition~\ref{def:approximation} directly generalizes to arbitrary cones for more than two objectives while the alternative characterization is universally applicable only in the biobjective case. A very general definition of approximation in multiobjective optimization with respect to arbitrary cones and a further characterization of when the two mentioned definition approaches are equivalent are given by Vanderpooten et al.~\cite{Vanderpooten+etal:covers+approximations}. They also present various results generalizing the following observation about approximations:
	
	\begin{observation}\label{obs:subset-approx}
		Consider $\gamma_1,\gamma_2 \in [\frac \pi 2, \pi]$, $\utheta_1 \in [0, \gamma_1]\setminus \{\gamma_1 - \pi, \frac \pi 2\}$, and $\utheta_2 \in [0, \gamma_2]\setminus \{\gamma_2 - \pi, \frac \pi 2\}$ such that $\utheta_1 \leq \utheta_2$ and $\otheta_1 \leq \otheta_2$, i.e., such that
		\begin{align*}
		C_{\gamma_1}^{\utheta_1} \subseteq  C_{\gamma_2}^{\utheta_2}.
		\end{align*}
		For $\alpha \geq 1$, if~$x' \in X$ is $\alpha$-approximated by~$x \in X$ in~$I_{\gamma_1}^{\utheta_1}$, then $x'$ also $\alpha$-approximated by $x$ in~$I_{\gamma_2}^{\utheta_2}$. Thus, any $\alpha$-approximation in~$I_{\gamma_1}^{\utheta_1}$ is an $\alpha$-approximation in~$I_{\gamma_2}^{\utheta_2}$. In particular, for any $\gamma \in [\frac \pi 2, \pi]$ and $\utheta \in [0, \gamma- \frac \pi 2]$, if~$x' \in X$ is $\alpha$-approximated by~$x \in X$ in~$I$ then $x'$ is also $\alpha$-approximated by $x$ in~$I_{\gamma}^{\utheta}$ and any $\alpha$-approximation in $I$ is an $\alpha$-approximation in~$I_\gamma^\utheta$. 
	\end{observation}
	
	\section{Structural Results}\label{sec:minimization}
	Observation~\ref{obs:subset-approx} states that, for any $\alpha \geq 1$, $\gamma \in [\frac \pi 2, \pi]$, and $\utheta \in [0, \gamma- \frac \pi 2]\setminus \{\gamma - \pi, \frac \pi 2\}$, any $\alpha$-approximation for~$I$ is also an $\alpha$-approximation for~$I_\gamma^\utheta$. 
	Vice versa, suppose that we can identify an approximation (or even the efficient set) for $I_\gamma^\utheta$ for some $\gamma \in (\frac \pi 2,\pi]$ and $\utheta \in [0, \gamma - \frac \pi 2]\setminus \{\gamma - \pi, \frac \pi 2\}$. Does this yield an $\alpha$-approximation for~$I$ for some~$\alpha$? It is easy to see that the answer to this question is ``no'' in general:
	\begin{example}\label{ex:onecone}
		Let $\alpha > 1$, $\gamma \in (\frac \pi 2, \pi]$, and $\utheta \in (0, \gamma - \frac \pi 2] \setminus \{\frac \pi 2\}$. Consider the following instance~$I$ of a biobjective minimization problem (see also Figure~\ref{fig:oneconeexample}):
		Let the feasible set consist of exactly two solutions~$x_1,x_2$ such that $f_1(x_1) = 1$, $f_2(x_1) = (\alpha-1) \cdot \tan \utheta$, $f_1(x_2) = \alpha$, and $f_2(x_2) = \frac{\alpha-1}{\alpha+1} \cdot \tan \utheta$.
		Then the efficient set of $I_\gamma^\utheta$ is~$\{x_1\}$, but $\{x_1\}$ is not an $\alpha$-approximation for $I$. However, $\{x_2\}$ is an $\alpha$-approximation for~$I$. 
	\end{example}
	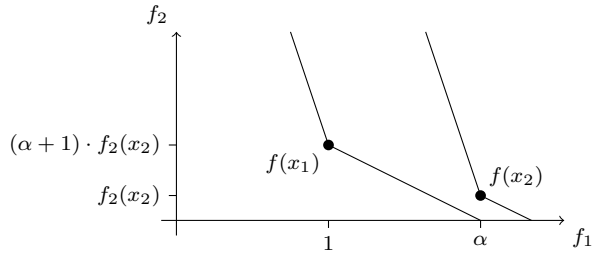
\begin{figure}[ht!]
		\begin{center}
			\begin{tikzpicture}[scale=1]
			
			
			
			
			\draw[->] (-0.2,0) -- (5.1,0) node[below right] {$f_1$};
			\draw[->] (0,-0.2) -- (0,2.5) node[above left] {$f_2$};

			\draw[-] (2,1) -- (4,0);
			\draw[-] (2,1) -- (1.5,2.5);
			
			\draw[-] (4,0.33) -- (4.67,0);
			\draw[-] (4,0.33) -- (4-13/18,2.5);

			\fill (2,1) node[below left]{$f(x_1)$} circle (2pt);
			\fill (4,0.33) node[above right]{$f(x_2)$} circle (2pt);

			\draw[-] (2,0) -- (2,-0.1) node[below] {$1$};
			\draw[-] (4,0) -- (4,-0.1) node[below] {$\alpha$};
			\draw[-] (0,0.33) -- (-0.1,0.33) node[left] {$f_2(x_2)$};
			\draw[-] (0,1) -- (-0.1,1) node[left] {$(\alpha +1) \cdot f_2(x_2)$};
			\end{tikzpicture}
			\caption{Illustration of Example~\ref{ex:onecone}. The solution~$x_2$ is not optimal with respect to~$\leqq_\gamma^\utheta$ and not $\alpha$-approximated by~$x_1$ (with respect to~$\leqq$).\label{fig:oneconeexample}}
		\end{center}
	\end{figure}
	We obtain the following proposition:
	\begin{proposition}\label{prop:onecone}
		For any $\gamma \in (\frac \pi 2,\pi]$ and $\utheta \in [0, \gamma - \frac \pi 2] \setminus \{\gamma - \pi, \frac \pi 2\}$, and any $\alpha \geq 1$, there exists an instance $I$ of a biobjective minimization problem such that the set of optimal solutions with respect to $\leqq_\gamma^\utheta$ is not an $\alpha$-approximation.
	\end{proposition}
	\begin{proof}
		For $\alpha > 1$ and $\utheta \neq 0$, the claim follows from Example~\ref{ex:onecone}. If $\alpha > 1$ and $\utheta = 0$, we have $\otheta \neq 0$, since $\gamma \neq \frac \pi 2$. Thus, we can simply exchange $f_1$ and $f_2$ and replace $\utheta$ by $\otheta$ in Example~\ref{ex:onecone} to obtain the claim. The claim for $\alpha = 1$ is a direct implication of the claim for any $\alpha > 1$. \qed
	\end{proof}
	
	Proposition~\ref{prop:onecone} states that the set of optimal solutions with respect to $\leqq_\gamma^\utheta$ for a single fixed pair of parameters~$(\gamma, \utheta)$ does not yield any approximation guarantee for~$I$. This is unsurprising: If the set of optimal solutions with respect to~$\leqq_\gamma^\utheta$ yielded any approximation guarantee, this would mean that, for the special case $\gamma = \pi$, where $I_\gamma^\utheta$ is a weighted sum scalarization of $I$, the (often unique) optimal solution of this scalarization would already yield an approximation guarantee in general.
	
	In the case $\gamma = \pi$, one is typically more interested in the set of supported solutions, i.e., the set of solutions that are optimal with respect to~$\leqq_\pi^\utheta$ for some (arbitrary)~$\utheta \in (0, \frac \pi 2)$. It is well-known that, for any biobjective minimization problem instance, the set of supported solutions is a $2$-approximation~\cite{Glasser+etal:multi-hardness}. We state this result using our terminology.
	\begin{theorem}[Gla{\ss}er et al.~\cite{Glasser+etal:multi-hardness}]\label{thm:ws-2-approx}
		For any biobjective minimization problem instance~$I$, let $X_W \subseteq X$ be a set of solutions that\textcolor{black}{, for any~$\utheta \in (0,\frac \pi 2)$, contains one optimal solution with respect to $\leqq_\pi^\utheta$.} Then $X_W$ is a $2$-approximation. 
	\end{theorem}
	
	Our goal is to generalize Theorem~\ref{thm:ws-2-approx} to arbitrary values of $\gamma \in [\frac \pi 2, \pi]$. More precisely, we want to obtain a result about the approximation guarantee achievable by solutions that are optimal with respect to~$\leqq_\gamma^\utheta$ for some fixed $\gamma \in [\frac \pi 2, \pi]$ but arbitrary $\utheta \in [0, \gamma- \frac \pi 2]\setminus \{\gamma - \pi, \frac \pi 2\}$. Example~\ref{ex:onesolution} shows that, for $\gamma \in (\frac \pi 2, \pi)$, it does not suffice to require a single arbitrary optimal solution for each~$\utheta$, as it is the case for~$\gamma = \pi$.
	
	\begin{example}\label{ex:onesolution}
		Let $\gamma \in (\frac \pi 2, \pi)$ and $\alpha \geq 1$. Consider the following instance of a biobjective minimization problem (see also Figure~\ref{fig:onesolutionexample}): The feasible set consists of exactly two solutions $x_1, x_2$ with $f_1(x_1) = \alpha + 1$, $f_2(x_1) = 1$, $f_1(x_2) = 1$, and $f_2(x_2) = \frac{-\cos \gamma}{\sin \gamma} \cdot (\alpha + 1) + 1$.
		
		Note that, for any $\utheta \in [0,\gamma-\frac \pi 2]$, we have $0 \leq \sin \utheta \leq - \cos \gamma$, where the first inequality is strict if $\utheta \neq 0$ and the second inequality is strict if $\utheta \neq \gamma - \frac \pi 2$, and we have $0 < \sin \gamma \leq \cos \utheta$, where, again, the second inequality is strict if  $\utheta \neq \gamma - \frac \pi 2$. Therefore, for any $\utheta \in [0,\gamma-\frac \pi 2]$, the following holds for the second objective function of  $I_\gamma^\utheta$:
		\begin{align*}
		\sin\utheta \cdot f_1(x_1) + \cos \utheta \cdot f_2(x_1) &= \sin \utheta \cdot  (\alpha + 1) + \cos \utheta\\
		&\leq \sin \utheta \cdot (\alpha + 1) + \cos \utheta + \sin \utheta\\
		& \leq (-\cos \gamma) \cdot (\alpha + 1) + \cos \utheta + \sin \utheta\\
		& \leq \frac{\cos \utheta}{\sin \gamma} \cdot  (-\cos \gamma)\cdot (\alpha + 1) + \cos \utheta + \sin \utheta\\
		&= \sin \utheta +\cos \utheta \cdot \left(\frac{- \cos \gamma}{\sin \gamma}\cdot (\alpha + 1) + 1\right)\\
		&=\sin\utheta \cdot f_1(x_2) + \cos \utheta \cdot f_2(x_2),
		\end{align*}
		where, if $\utheta \neq 0$, the first inequality is strict, and, if $\utheta \neq \gamma - \frac \pi 2$, the second and third inequalities are strict. Thus, $x_1$ is optimal with respect to~$\leqq_\gamma^\utheta$ for any $\utheta \in [0, \gamma- \frac \pi 2]$. On the other hand, $x_1$ does not $\alpha$-approximate $x_2$ (with respect to~$\leqq$).
	\end{example}
	\begin{figure}[ht!]
		\begin{center}
			\begin{tikzpicture}[scale=1]
			
			
			
			
			\draw[->] (-0.2,0) -- (5.2,0) node[below right] {$f_1$};
			\draw[->] (0,-0.2) -- (0,5.1) node[above left] {$f_2$};

			\draw[-] (1,3.4) -- (1,5);
			\draw[-] (1,3.4) -- (5,0.2);
			
			\draw[dotted] (1,3.4) -- (5,3.4);
			\draw[dotted] (1,3.4) -- (0,4.65);
			
			\draw[dashed] (1,3.4) -- (0.44,5);
			\draw[dashed] (1,3.4) -- (5,2);

			\fill (3,1) node[below left]{$f(x_1)$} circle (2pt);
			\fill (1,3.4) node[below left]{$f(x_2)$} circle (2pt);
			
			\draw[dashed] (1.8,3.4) arc (0:-19.33:0.8);
			\draw[dashed] (1.8,3.4) arc (0:109.33:0.8);
			\draw[] (1.7,3.4) arc (0:-38.66:0.7);
			\draw[] (1.7,3.4) arc (0:90:0.7);
			\draw[dotted] (1.9,3.4) arc (0:128.66:0.9);
			
			\node[] at (1.25,3.65)  {$\gamma$};
			
			\draw[-] (1,0) -- (1,-0.1) node[below] {$1$};
			\draw[-] (3,0) -- (3,-0.1) node[below] {$\alpha + 1$};
			\draw[-] (0,1) -- (-0.1,1) node[left] {$f_2(x_1)$};
			\draw[-] (0,3.4) -- (-0.1,3.4) node[left] {$f_2(x_2)$};
			\end{tikzpicture}
			\caption{Illustration of Example~\ref{ex:onesolution}. The dominance cone of $f(x_2)$ in $I_\gamma^\utheta$ is illustrated for $\utheta = 0$ (dotted), $\utheta = \mtheta$ (dashed), and $\utheta = \gamma - \frac \pi 2$ (solid). The solution~$x_1$ is not dominated by $x_2$ and is, thus, optimal with respect to $\leqq_\gamma^\utheta$ for any $\utheta \in [0, \gamma- \frac \pi 2]$. \label{fig:onesolutionexample}}
		\end{center}
	\end{figure}
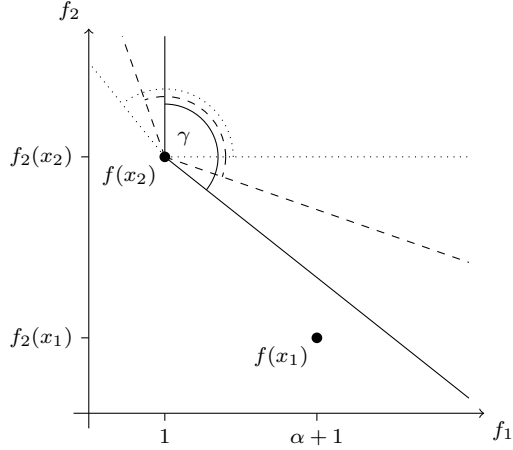

	We now generalize Theorem~\ref{thm:ws-2-approx} to arbitrary values of $\gamma$. We will see that, for any~$\gamma \in [\frac \pi 2, \pi]$, the set of $\gamma$-supported solutions is an approximation. The approximation guarantee obtained from our result is equal to 1 for $\gamma = \frac \pi 2$, is equal to 2 for $\gamma = \pi$, and, interestingly,  increases continuously in between depending on $\gamma$.
	
	Moreover, for any $\gamma \in (\frac \pi 2, \pi]$ and $\utheta \in [0, \gamma- \frac \pi 2]\setminus \{\gamma - \pi, \frac \pi 2\}$, we provide a weighted max-ordering scalarization of~$I_\gamma^\utheta$ such that, for fixed $\gamma$, a set containing only one optimal solution of this scalarization for each $\utheta$ yields the same approximation guarantee (analogous to Threorem~\ref{thm:ws-2-approx}). For $\gamma = \pi$, this scalarization\footnote{We use the term ``scalarization''  here even though a scalarization of a single objective problem does not really deserve the name.} naturally yields the (single objective) instance itself, so, this result is indeed a generalization of Theorem~\ref{thm:ws-2-approx}. We further generalize this result to approximate solutions of the provided scalarization.
	
	First, note the following simple property of weighted max-ordering scalarizations:
	
	\begin{lemma}\label{lem:max-ordering-approx}
		Let $I = (X,f)$ be a biobjective minimization problem instance, let $\alpha \geq 1$, and let $w_1, w_2 > 0$ be given. Let $x \in X$ be an $\alpha$-approximate solution for the weighted max-ordering scalarization of~$I$ with weights $w_1,w_2$ and let $x' \in X$ be a solution such that $w_1 \cdot f_1(x') = w_2 \cdot f_2(x')$.
		Then $x'$ is $\alpha$-approximated by $x$ in~$I$. 
	\end{lemma}
	\begin{proof}
		In the first component, we have
		\begin{align*}
		w_1 \cdot f_1(x) & \leq \max\{w_1 \cdot f_1(x),w_2 \cdot f_2(x) \}\\
		&\leq \alpha \cdot \max\{w_1 \cdot f_1(x'),w_2 \cdot f_2(x') \}\\
		& = \alpha \cdot w_1 \cdot f_1(x').
		\end{align*}
		The approximation guarantee in the second component follows analogously. \qed
	\end{proof}
	
	The following lemma states that a solution $x \in X$ that approximates another solution $x' \in X$ with respect to $\leqq_\gamma^\utheta$ for some $\gamma$ and $\utheta$ also approximates~$x'$ with respect to $\leqq$ by some factor. This factor depends on  $\gamma$, $\utheta$, and~$f(x')$. Note that, by Proposition~\ref{prop:onecone}, we cannot expect this factor to depend solely on~$\gamma$ and~$\utheta$.
	
	\begin{lemma}\label{lem:general-approx}
		Let $\gamma \in [\frac \pi 2, \pi]$, $\utheta \in [0, \gamma- \frac \pi 2]\setminus \{\gamma - \pi, \frac \pi 2\}$, $\alpha \geq 1$, and let $I = (X,f)$ be a biobjective minimization problem instance. Let $x' \in X$ be $\alpha$-approximated by $x \in X$ with respect to $\leqq_\gamma^\utheta$. Then $x$ approximates $x'$ (with respect to $\leqq$) with factor
		\begin{align*}
		\alpha \cdot \left(1+ \max\left\{\frac{f_1(x')}{f_2(x')} \cdot \tan \utheta, \frac{f_2(x')}{f_1(x')} \cdot \tan \otheta\right\}\right).
		\end{align*}
	\end{lemma}
	\begin{proof}
		In the first component, we obtain
		\begin{align*}
		f_1(x) &\leq \frac 1 {\cos \otheta} \cdot \left(\cos \otheta \cdot f_1(x) + \sin \otheta \cdot f_2(x)\right)\\
		& \leq \frac 1 {\cos \otheta}  \cdot \alpha \cdot \left(\cos \otheta \cdot f_1(x') + \sin \otheta \cdot f_2(x')\right)\\
		&= \alpha \cdot \left(1+\tan \otheta \cdot \frac{f_2(x')}{f_1(x')}\right) \cdot f_1(x').
		\end{align*}
		Similarly, in the second component, we obtain
		\begin{align*}	
		f_2(x) &\leq \frac 1 {\cos \utheta} \cdot \left(\sin \utheta \cdot f_1(x) + \cos \utheta \cdot f_2(x)\right)\\
		& \leq \frac 1 {\cos \utheta}  \cdot \alpha \cdot \left(\sin \utheta \cdot f_1(x') + \cos \utheta \cdot f_2(x')\right)\\
		&= \alpha \cdot \left(1+\tan \utheta \cdot \frac{f_1(x')}{f_2(x')}\right) \cdot f_2(x').
		\end{align*}
		This immediately yields the claimed approximation guarantee.
		\qed
	\end{proof}
	
	The next lemma states that, for $\gamma \in (\frac \pi 2,\pi]$, $\utheta \in (0, \gamma - \frac \pi 2)$, and an instance $I = (X,f)$, if we use the weights $w_1 = \frac{\sqrt{\sin \utheta}}{\sqrt{\cos \otheta}} + \frac{\sqrt{\cos \utheta}}{\sqrt{\sin \otheta}}$ and $w_2 = \frac{\sqrt{\sin \otheta}}{\sqrt{\cos \utheta}} + \frac{\sqrt{\cos \otheta}}{\sqrt{\sin \utheta}}$ for a weighted max-ordering scalarization of $I_\gamma^\utheta$, then any solution $x' \in X$ for which $\frac{f_1(x')}{f_2(x')}  =\frac{\sqrt{\tan\otheta}}{\sqrt{\tan\utheta}}$ meets the conditions of Lemma~\ref{lem:max-ordering-approx}. This scalarization is illustrated in Figure~\ref{fig:conepush}.
	
	\begin{figure}[ht!]
		\begin{center}
			\begin{tikzpicture}[scale=1]
			\draw[fill,gray!45] (0,0) -- (4,0) -- (3,2) -- (0,8/3)-- (0,0);
			\draw[->] (-0.2,0) -- (7.1,0) node[below right] {$f_1$};
			\draw[->] (0,-0.2) -- (0,4.5) node[above left] {$f_2$};

			\draw[dashed] (0,0) -- (6,4);
			
			\draw[-] (3,2) -- (4,0);
			\draw[-] (3,2) -- (0,8/3);
			
			\draw[dashed] (1.5,1) -- (1.5,0);
			\draw[dashed] (1.5,1) -- (0,1);
			
			\draw[dashed] (3,2) -- (3,0);
			\draw[dashed] (3,2) -- (0,2);

			\fill (1.5,7/3) node[above]{$f(x)$} circle (3pt);
			\fill (1,3.4) circle (2pt);
			\fill (2,3.4) circle (2pt);
			\fill (2.2,4.2) circle (2pt);
			\fill (3.1,2.7) circle (2pt);
			\fill (3.8,1.6) circle (2pt);
			\fill (3.8,4.1) circle (2pt);
			\fill (4.1,0.8) circle (2pt);
			\fill (4.2,3.3) circle (2pt);
			\fill (4.4,2.5) circle (2pt);
			\fill (4.5,1.8) circle (2pt);
			\fill (5.3,2.8) circle (2pt);
			\fill (5.9,1) circle (2pt);
			
			\draw (1.8,2) arc (180:168.13:1.2);
			\node[] at (1.9,1.8)  {$\utheta$};
			\draw[-] (2.05,2.1) -- (1.95,1.95);
			
			\draw (3,1) arc (270:296.57:1);
			\node[] at (3.75,1.3)  {$\otheta$};
			\draw[-] (3.2,1.2) -- (3.5,1.3);
			
			\draw (2.4,2) arc (180:168.13:0.6);
			\draw (2.4,2) arc (180:296.57:0.6);
			\node[] at (2.85,1.7)  {$\gamma$};

			\draw[-] (1.5,0.1) -- (1.5,-0.1) node[below] {$\sqrt{\tan \otheta}$};
			\draw[-] (0.1,1) -- (-0.1,1) node[left] {$\sqrt{\tan \utheta}$};
			\draw[-] (3,0.1) -- (3,-0.1) node[below] {$c_1$};
			\draw[-] (0.1,2) -- (-0.1,2) node[left] {$c_2$};
			\draw[-] (4,0.1) -- (4,-0.1) node[below] {$d_1$};
			\draw[-] (0.1,8/3) -- (-0.1,8/3) node[left] {$d_2$};
			\end{tikzpicture}
			\caption{Illustration of the weighted max-ordering scalarization of $I_\gamma^\utheta$ with weights  $w_1 = \frac{\sqrt{\sin \utheta}}{\sqrt{\cos \otheta}} + \frac{\sqrt{\cos \utheta}}{\sqrt{\sin \otheta}}$ and $w_2 = \frac{\sqrt{\sin \otheta}}{\sqrt{\cos \utheta}} + \frac{\sqrt{\cos \otheta}}{\sqrt{\sin \utheta}}$ for given $\gamma \in (\frac \pi 2, \pi]$ and $\utheta \in (0, \gamma - \frac \pi 2)$. The solution~$x$ is optimal for this scalarization so there does not exist any feasible point in the gray region. For this choice of weights, we have $\frac{d_1}{c_1} = \frac{d_2}{c_2} = 1 + \sqrt{\tan \utheta} \cdot \sqrt{\tan \otheta}$ (see Proposition~\ref{prop:Q-approx}).  \label{fig:conepush}}
		\end{center}
	\end{figure}
	
	\begin{lemma}\label{lem:right-weights}
		Let  $\gamma \in (\frac \pi 2, \pi]$, $\utheta \in (0, \gamma - \frac \pi 2)$, and let $I = (X,f)$ be a biobjective minimization problem instance. Let $x' \in X$ such that $\frac{f_1(x')}{f_2(x')}  = \frac{\sqrt{\tan\otheta}}{\sqrt{\tan\utheta}}$.
		Moreover, let $w_1 = \frac{\sqrt{\sin \utheta}}{\sqrt{\cos \otheta}} + \frac{\sqrt{\cos \utheta}}{\sqrt{\sin \otheta}}$ and $w_2 = \frac{\sqrt{\sin \otheta}}{\sqrt{\cos \utheta}} + \frac{\sqrt{\cos \otheta}}{\sqrt{\sin \utheta}}$.
		Then
		\begin{align*}
		w_1 \cdot \left(\cos \otheta \cdot f_1(x') + \sin \otheta \cdot f_2(x') \right) = 
		w_2 \cdot \left(\sin \utheta \cdot f_1(x')  + \cos \utheta \cdot f_2(x' )\right).
		\end{align*}
	\end{lemma}
	\begin{proof}
		We know that $f_1(x') = \sqrt{\tan \otheta} \cdot \frac{f_2(x')}{\sqrt{\tan \utheta}}$, so it suffices to show that
		\begin{align*}
		w_1 \cdot \left(\cos \otheta \cdot \sqrt{\tan \otheta}+ \sin \otheta \cdot \sqrt{\tan \utheta} \right) = 
		w_2 \cdot \left(\sin \utheta \cdot \sqrt{\tan \otheta}  + \cos \utheta \cdot \sqrt{\tan \utheta}\right).
		\end{align*}
		Using the definition of $w_1,w_2$ and that $\tan = \frac \sin \cos$, this is a simple computation:
		\begin{align*}
		& \left(\frac{\sqrt{\sin \utheta}}{\sqrt{\cos \otheta}} + \frac{\sqrt{\cos \utheta}}{\sqrt{\sin \otheta}}\right) \cdot \left(\cos \otheta \cdot \sqrt{\tan \otheta}+ \sin \otheta \cdot \sqrt{\tan \utheta} \right)\\
		=& \; 2 \cdot \sqrt{\sin \utheta} \cdot \sqrt{\sin \otheta} + \sqrt{\cos \utheta} \cdot \sqrt{\cos \otheta} + \frac{\sin \utheta \cdot \sin \otheta}{\sqrt{\cos \utheta} \cdot \sqrt{\cos \otheta}}\\
		=& \; \left(\frac{\sqrt{\sin \otheta}}{\sqrt{\cos \utheta}} + \frac{\sqrt{\cos \otheta}}{\sqrt{\sin \utheta}}\right) \cdot \left(\sin \utheta \cdot \sqrt{\tan \otheta}+ \cos \utheta \cdot \sqrt{\tan \utheta} \right).
		\end{align*}\qed	
	\end{proof}
	The following proposition combines Lemma~\ref{lem:max-ordering-approx}, Lemma~\ref{lem:general-approx}, and Lemma~\ref{lem:right-weights}. It first states that, for $x'$, $\utheta$, $w_1$, and $w_2$ as in Lemma~\ref{lem:right-weights}, we can approximate~$x'$ not only in the corresponding weighted max-ordering scalarization but also in~$I_\gamma^\utheta$. Then it states that we even obtain an approximation factor for $I$. We will see that, for any solution~$x' \in X$, the angle $\utheta$ satisfying $\frac{f_1(x')}{f_2(x')}  = \frac{\sqrt{\tan\otheta}}{\sqrt{\tan\utheta}}$ corresponds to $x'$ in the sense that, in the maximum in the approximation factor provided in Lemma~\ref{lem:general-approx}, both terms are equal (a geometric explanation for this is given in Figure~\ref{fig:conepush}). Therefore, the approximation factor obtained for~$I$ depends only on~$\gamma$ and~$\utheta$ and does not involve a maximum.
	
	\begin{proposition}\label{prop:Q-approx}
		Let $\alpha \geq 1$, $\gamma \in (\frac \pi 2, \pi]$, $\utheta \in (0, \gamma- \frac \pi 2)$, and let $I = (X,f)$ be a biobjective minimization problem instance. For a solution~~$x \in X$ that is $\alpha$-approximate for the weighted max-ordering scalarization of $I_\gamma^\utheta$ with weights $w_1 = \frac{\sqrt{\sin \utheta}}{\sqrt{\cos \otheta}} + \frac{\sqrt{\cos \utheta}}{\sqrt{\sin \otheta}}$ and $w_2 = \frac{\sqrt{\sin \otheta}}{\sqrt{\cos \utheta}} + \frac{\sqrt{\cos \otheta}}{\sqrt{\sin \utheta}}$, any solution~$x' \in X$ with
		\begin{align}\frac{f_1(x')}{f_2(x')} = \frac{\sqrt{\tan\otheta}}{\sqrt{\tan\utheta}} \label{eq:q-definition}\end{align}
		\begin{enumerate}[(i)]
			\item is $\alpha$-approximated by $x$ with respect to~$\leqq_\gamma^\utheta$, and
			\item is $\left(\alpha \cdot \left(1+\sqrt{\tan \utheta} \cdot \sqrt{\tan \otheta}\right)\right)$-approximated by~$x$ (with respect to~$\leqq$).
		\end{enumerate}
	\end{proposition}
	\begin{proof}
		We first prove (i). Lemma~\ref{lem:right-weights} implies that 
		\begin{align*}
		w_1 \cdot \left(\cos \otheta \cdot f_1(x') + \sin \otheta \cdot f_2(x') \right) = 
		w_2 \cdot \left(\sin \utheta \cdot f_1(x')  + \cos \utheta \cdot f_2(x' )\right).
		\end{align*}
		Thus, we can apply Lemma~\ref{lem:max-ordering-approx} to the weighted max-ordering scalarization of~$I_\gamma^\utheta$ with weights $w_1,w_2$, which immediately yields that~$x'$ is $\alpha$-approximated by~$x$ with respect to~$\leqq_\gamma^\utheta$.
		
		In order to prove (ii), we apply Lemma~\ref{lem:general-approx} to obtain that $x'$ is approximated by $x$ with factor
		\begin{align*}
		\alpha \cdot \left(1+ \max\left\{\frac{f_1(x')}{f_2(x')} \cdot \tan \utheta, \frac{f_2(x')}{f_1(x')} \cdot \tan \otheta\right\}\right).
		\end{align*}
		Since~\eqref{eq:q-definition} holds, we know that
		\begin{align*}
		\max\left\{\frac{f_1(x')}{f_2(x')} \cdot \tan \utheta, \frac{f_2(x')}{f_1(x')}\cdot \tan \otheta\right\} &= \max\left\{\frac{\sqrt{\tan\otheta}}{\sqrt{\tan\utheta}} \cdot \tan \utheta, \frac{\sqrt{\tan\utheta}}{\sqrt{\tan\otheta}}\cdot \tan \otheta\right\}\\
		&= \sqrt{\tan\utheta} \cdot \sqrt{\tan\otheta},
		\end{align*}
		which yields~(ii). \qed
	\end{proof}
	Proposition~\ref{prop:Q-approx} states that, for given $\gamma \in (\frac \pi 2, \pi]$, any solution $x' \in X$ can be approximated by a solution that is $\alpha$-approximate for a specific weighted max-ordering scalarization of $I_\gamma^\utheta$, if $\utheta$ is chosen such that~\eqref{eq:q-definition} holds. The achievable approximation factor depends on $\gamma$ and $\utheta$. The following lemma provides an upper bound on this approximation factor that solely depends on $\gamma$. Its proof is given in Appendix~A.
	\begin{lemma}\label{lem:utheta-otheta-leq-mtheta}
		Let $\gamma \in [\frac \pi 2, \pi]$ and  $\utheta \in [0, \gamma- \frac \pi 2]\setminus \{\gamma - \pi, \frac \pi 2\}$. Then we have $\sqrt{\tan \utheta} \cdot \sqrt{\tan \otheta} \leq \tan \mtheta$,	\textcolor{black}{where $\mtheta = \frac \gamma 2 - \frac \pi 4$.}
	\end{lemma}
	
	\noindent We are now ready to prove our main result.
	\begin{theorem}\label{thm:cone-approx}
		Let $I= (X,f)$ be a biobjective minimization problem instance and let $\gamma \in (\frac \pi 2, \pi]$. Let $X_Q \subseteq X$ be a set of solutions that, for any~$\utheta \in (0,\gamma -\frac \pi 2)$, contains an $\alpha$-approximate solution for the weighted max-ordering scalarization of~$I_\gamma^\utheta$ with weights $w_1 = \frac{\sqrt{\sin \utheta}}{\sqrt{\cos \otheta}} + \frac{\sqrt{\cos \utheta}}{\sqrt{\sin \otheta}}$ and $w_2 = \frac{\sqrt{\sin \otheta}}{\sqrt{\cos \utheta}} + \frac{\sqrt{\cos \otheta}}{\sqrt{\sin \utheta}}$. Then $X_Q$ is an $\left(\alpha \cdot (1+\tan \mtheta)\right)$-approximation (for~$I$), where $\mtheta = \frac \gamma 2 - \frac \pi 4$.
	\end{theorem}
	\begin{proof}
		Let $x' \in X$ be any feasible solution. Choose $\utheta \in (0, \gamma- \frac \pi 2)$ such that $\frac{\sqrt{\tan \utheta}}{\sqrt{\tan \otheta}} = \frac{f_1(x')}{f_2(x')}$, i.e.,
		$\utheta = \arctan\left(\frac 1 q \cdot \left(s \cdot \tan \gamma + \sqrt{1+ s^2 \cdot (\tan \gamma)^2}\right)\right)$ for $q = \frac{f_1(x')}{f_2(x')}$ and $s = \frac 1 2 \cdot \left(q + \frac 1 q\right)$. Then $X_Q$ contains an $\alpha$-approximate solution for the weighted max-ordering scalarization of $I_\gamma^\utheta$ with weights $w_1 = \frac{\sqrt{\sin \utheta}}{\sqrt{\cos \otheta}} + \frac{\sqrt{\cos \utheta}}{\sqrt{\sin \otheta}}$ and $w_2 = \frac{\sqrt{\sin \otheta}}{\sqrt{\cos \utheta}} + \frac{\sqrt{\cos \otheta}}{\sqrt{\sin \utheta}}$. Proposition~\ref{prop:Q-approx} states that $x'$ is $\left(\alpha \cdot \left(1+\sqrt{\tan \utheta} \cdot \sqrt{\tan \otheta}\right)\right)$-approximated by~$x$. Thus, by Lemma~\ref{lem:utheta-otheta-leq-mtheta}, $x'$ is also $\left(\alpha \cdot(1+ \tan \mtheta)\right)$-approximated by~$x$.
		\qed
	\end{proof}
	
	Note that one can obtain Theorem~\ref{thm:ws-2-approx} by setting $\gamma = \pi$ and $\alpha = 1$ in Theorem~\ref{thm:cone-approx}. Thus, Theorem~\ref{thm:cone-approx} is indeed a generalization of Theorem~\ref{thm:ws-2-approx}.
	
	The following corollary collects several alternative formulas expressing the approximation factor $(\alpha \cdot (1+\tan \mtheta))$ obtained in Theorem~\ref{thm:cone-approx}. Its proof is given in Appendix~B.
	
	\begin{corollary}\label{cor:alternative-formulas}
		The set $X_Q$ from Theorem~\ref{thm:cone-approx} is an $\left(\alpha \cdot (1+ S)\right)$-approximation, where 
		\begin{align*}
		S = \tan\left(\frac{\gamma - \frac \pi 2}{2}\right) = \frac{1 -\sin \gamma}{-\cos \gamma} = \frac{- \cos \gamma}{1+\sin \gamma} = \tan \gamma + \sqrt{1+ (\tan \gamma)^2}.
		\end{align*}
	\end{corollary}

	Theorem~\ref{thm:cone-approx} yields the following corollary. It provides the approximation factor achievable by the set of $\gamma$-supported solutions in a biobjective minimization problem instance for any inner angle $\gamma \in [\frac \pi 2, \pi]$. Of course, the set of $\frac \pi  2$-supported solutions, i.e., the efficient set, is a $1$-approximation and the set of ($\pi$-) supported solutions is a $2$-approximation. In between $\frac \pi 2$ and $\pi$, the approximation factor is continuous and strictly increasing in~$\gamma$. See Figure~\ref{fig:approx-graph} for an illustration. 
	
	\begin{corollary}\label{cor:gamma-supp-approx}
		For any biobjective minimization problem instance and any $\gamma \in [\frac \pi 2, \pi]$, the set of $\gamma$-supported solutions is a $(1+\tan \mtheta)$-approximation, where~$\mtheta = \frac \gamma 2 - \frac \pi 4$.
	\end{corollary}
	\begin{proof}
		For $\gamma = \frac \pi 2$, the claim is obviously true as the set of efficient solutions is a $1$-approximation. For $\gamma \in (\frac \pi 2, \pi]$, we know that, for any $\utheta \in (0, \gamma - \frac \pi 2)$ and any weighted max-ordering scalarization of $I_\gamma^\utheta$, there exists a solution that is optimal for both the weighted max-ordering scalarization of $I_\gamma^\utheta$ and for $I_\gamma^\utheta$ itself, and is therefore also $\gamma$-supported. Thus, the set of $\gamma$-supported solutions contains an optimal solution for any weighted max-ordering scalarization of $I_\gamma^\utheta$ for any $\utheta \in (0, \gamma - \frac \pi 2)$. The claim follows from Theorem~\ref{thm:cone-approx} setting $\alpha = 1$.
		\qed	
	\end{proof}
	\begin{figure}
		\begin{center}
			\begin{tikzpicture}
			\begin{axis}[ 
			xlabel={inner angle $\gamma$},
			ylabel={approximation factor},
			xtick = {pi/2,pi},
			xticklabels = {$\frac\pi 2$,$\pi$},
			] 
			\addplot[domain = pi/2:pi] {1+tan((deg(x)-90)/2))};
			\addplot[domain = pi/2:pi, dashed] {deg(x)/90};  
			\end{axis}
			\end{tikzpicture}
			\caption{Approximation factor achieved by the set of $\gamma$-supported solutions due to Corollary~\ref{cor:gamma-supp-approx} (solid) and Corollary~\ref{cor:rule-of-thumb} (dashed).\label{fig:approx-graph}}
		\end{center}
	\end{figure}
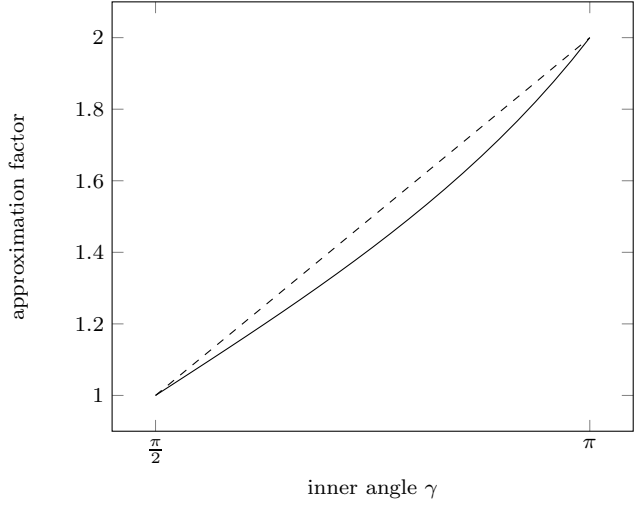
	
	Figure~\ref{fig:approx-graph} shows that the increase of the approximation factor achieved by the set of $\gamma$-supported solutions is quite close to linear in $\gamma$. In fact, it is slightly convex. Thus, a reasonable rule of thumb is that the percentage at which the angle $\gamma$ is between $\frac \pi 2$ and $\pi$ is the approximation accuracy that is lost by the set of $\gamma$-supported solutions compared to the efficient set. The next corollary formalizes this rule of thumb.
	
	\begin{corollary}\label{cor:rule-of-thumb}
		For any biobjective minimization problem instance and any $\gamma \in [\frac \pi 2, \pi]$, the set of $\gamma$-supported solutions is a $\frac{2\gamma}{\pi}$-approximation.
	\end{corollary}
	\begin{proof}
		Note that, since $\tan$ is a convex function on $[0,\frac \pi 4]$, where $\tan 0 = 0$ and $\tan \frac \pi 4 = 1$, we have
		\begin{equation*}
		\tan \mtheta = \tan\left(\frac {4 \cdot \mtheta}{\pi} \cdot \frac \pi 4 \right)\leq \frac {4 \cdot \mtheta}{\pi} \cdot \tan \frac \pi 4
		= \frac {4 \cdot \mtheta}{\pi}.
		\end{equation*}
		Thus, by Corollary~\ref{cor:gamma-supp-approx}, the set of $\gamma$-supported solutions is a $\left(1 +\frac {4 \cdot \mtheta}{\pi}\right)$-approximation, where	
		\textcolor{black}{$1 +\frac {4 \cdot \mtheta}{\pi} = 1 + \frac{2 \gamma - \pi}{\pi} = \frac{2\gamma}{\pi}$}.
		\qed	
	\end{proof}
	The following proposition states that Theorem~\ref{thm:cone-approx} and Corollary~\ref{cor:gamma-supp-approx} are tight in the sense that, for any inner angle $\gamma$ (including the cases $\gamma = \frac \pi 2$ and $\gamma = \pi$), no better approximation guarantee than the one provided is achievable by approximations with respect to~$\leqq_\gamma^\utheta$ for all~$\utheta \in [0,\gamma-\frac \pi 2]\setminus \{\gamma - \pi, \frac \pi 2\}$.
	\begin{proposition}\label{prop:tightness}
		For any $\gamma \in [\frac \pi 2,\pi]$, any $\alpha \geq 1$, and any $\varepsilon > 0$, there exists an instance~$I = (X,f)$ of a biobjective optimization problem for which a set that is an $\alpha$-approximation with respect to $\leqq_\gamma^\utheta$ for all~$\utheta \in [0,\gamma-\frac \pi 2]\setminus \{\gamma - \pi, \frac \pi 2\}$ is not an \textcolor{black}{$\left(\alpha \cdot (1+\tan \mtheta) - \varepsilon\right)$-approximation} with respect to $\leqq$.
	\end{proposition}
	\begin{proof}
		Define $\varepsilon' > 0$ such that $\varepsilon' < \min\{\frac \varepsilon \alpha,1\}$.
		Consider the following instance~$I$, which is illustrated in Figure~\ref{fig:tightness}: Let the feasible set consist of exactly three solutions, $x_1,x_2,x_3$ such that
		\begin{align*}
		&f_1(x_1) = \alpha \cdot \left(1+ (1-\varepsilon') \cdot \tan \mtheta \right),&& \qquad f_2(x_1) = \alpha \cdot \varepsilon',  \\
		&f_1(x_2) = \alpha \cdot \varepsilon', && \qquad f_2(x_2) = \alpha \cdot \left(1+ (1-\varepsilon') \cdot \tan \mtheta \right),\\
		&f_1(x_3) = 1, && \qquad f_2(x_3) = 1.
		\end{align*}
		Then $\{x_1,x_2\}$ is an $\alpha$-approximation with respect to $\leqq_\gamma^\utheta$ for all $\utheta \in [0,\gamma-\frac \pi 2]\setminus \{\gamma - \pi, \frac \pi 2\}$:
		For $\utheta \leq \mtheta$, we have $\otheta \geq \mtheta$ and, therefore, $\tan \utheta \leq \tan \mtheta \leq \tan \otheta$. We can compute
		\begin{align*}
		\cos \otheta \cdot f_1(x_1)  + \sin \otheta \cdot f_2(x_1) &= \cos \otheta \cdot \alpha \cdot \left(1+ (1-\varepsilon') \cdot \tan \mtheta \right) + \sin \otheta \cdot \alpha \cdot \varepsilon'\\
		&\leq  \cos \otheta \cdot \alpha \cdot \left(1+ (1-\varepsilon') \cdot \tan \otheta \right) + \sin \otheta \cdot \alpha \cdot \varepsilon'\\
		&= \alpha \cdot \left(\cos \otheta + \sin \otheta\right)\\
		&=  \alpha \cdot \left(\cos \otheta \cdot f_1(x_3)  + \sin \otheta \cdot f_2(x_3) \right)
		\end{align*}
		and, since $\mtheta \leq \frac \pi 4$ and, therefore, $\tan \mtheta \leq \tan \frac \pi 4 = 1$,
		\begin{align*}
		\sin \utheta \cdot f_1(x_1) + \cos \utheta \cdot f_2(x_1) &= \sin \utheta \cdot \alpha \cdot \left(1+ (1-\varepsilon') \cdot \tan \mtheta \right) + \cos \utheta \cdot \alpha \cdot \varepsilon'\\	
		&\leq \sin \utheta \cdot \alpha \cdot \left(1+ (1-\varepsilon') \cdot \frac 1 {\tan \mtheta} \right) + \cos \utheta \cdot \alpha \cdot \varepsilon'\\
		&\leq \sin \utheta \cdot \alpha \cdot \left(1+ (1-\varepsilon') \cdot \frac 1 {\tan \utheta} \right) + \cos \utheta \cdot \alpha \cdot \varepsilon'\\	
		&= \alpha \cdot \left(\sin \utheta + \cos \utheta\right)\\
		&=  \alpha \cdot \left(\sin \utheta \cdot f_1(x_3)  + \cos \utheta \cdot f_2(x_3) \right).
		\end{align*}
		Thus, for $\utheta \leq \mtheta$, $x_3$ is $\alpha$-approximated by $x_1$ with respect to $\leqq_\gamma^\utheta$. Similarly, we can prove that, for $\utheta \geq \mtheta$, $x_3$ is $\alpha$-approximated by~$x_2$ with respect to~$\leqq_\gamma^\utheta$.
		
		However, $\{x_1,x_2\}$ is not an $\left(\alpha \cdot \left(1+\tan \mtheta\right) - \varepsilon\right)$-approximation (with respect to~$\leqq$): We have $\tan \mtheta \leq 1$ and, thus, 
		\begin{align*}
		\left(\alpha \cdot \left(1+\tan \mtheta\right) - \varepsilon\right) \cdot f_1(x_3) &< \alpha \cdot \left(1+ \tan \mtheta - \varepsilon'\right)\\
		& \leq \alpha \cdot \left(1+ \tan \mtheta - \varepsilon' \cdot \tan \mtheta \right)\\
		&= f_1(x_1).
		\end{align*}
		Similarly, we have
		\begin{align*}
		\left(\alpha \cdot (1+\frac{- \cos \gamma}{\sin \gamma + 1}) - \varepsilon\right) \cdot f_2(x_3) < f_2(x_2).
		\end{align*}
		Thus, $x_3$ is not $\left(\alpha \cdot \left(1+\frac{- \cos \gamma}{\sin \gamma + 1}\right) - \varepsilon\right)$-approximated.
		\qed
	\end{proof}
	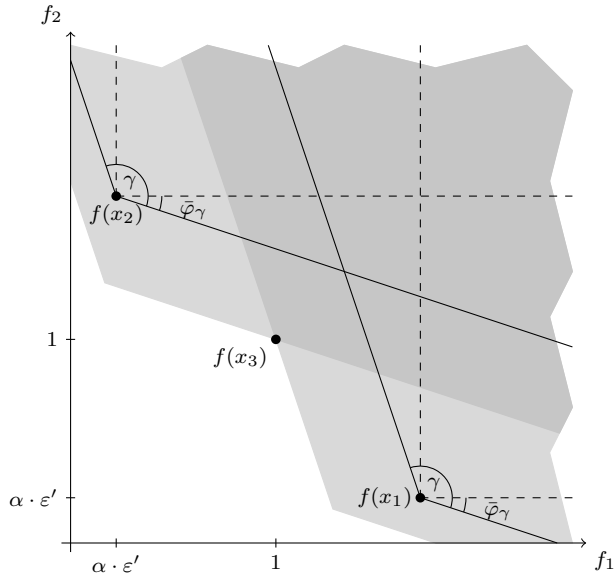
\begin{figure}[ht!]
		\begin{center}
			\begin{tikzpicture}[scale=0.6]

			\draw[fill,gray!30] (4.5,4.5) -- (23/3*3/4,3/4) -- (32/3*3/4,0)-- (32/3,0) -- (23/3,1) -- (6,6) --(1,23/3)-- (0,32/3)-- (0,32/3*3/4)--(3/4,23/3*3/4) -- (4.5,4.5);
			
			\draw[fill,gray!45] (4.5,4.5) -- (15/4,27/4) -- (6,6)-- (27/4,15/4) -- (4.5,4.5);	
			
			\draw[fill,gray!30] (32/3*3/4,0)-- (32/3,0) -- (23/3,1) -- (6,6) --(1,23/3)-- (0,32/3)-- (0,32/3*3/4) --(0,11) -- (2,10.5) --(3,11) -- (5,10.5)-- (6,11) -- (8,10.5) -- (10,11) -- (11,10) -- (10.5,8) -- (11,6) -- (10.5,5)--(11,3) --  (10.5,2) -- (11,0) -- (32/3*3/4,0);
			
			\draw[fill,gray!45] (6,6) -- (49/11,117/11) -- (3,11) -- (17/7,75/7)-- (15/4,27/4) -- (6,6);
			
			\draw[fill,gray!45] (6,6) -- (117/11,49/11) -- (11,3) -- (75/7,17/7)-- (27/4,15/4) -- (6,6);

			\draw[fill,gray!45] (6,6)-- (49/11,117/11) -- (5,10.5)-- (6,11) -- (8,10.5) -- (10,11) -- (11,10) -- (10.5,8) -- (11,6) -- (10.5,5)--(117/11,49/11)-- (6,6);

			\draw[->] (-0.2,0) -- (11.3,0) node[below right] {$f_1$};
			\draw[->] (0,-0.2) -- (0,11.3) node[above left] {$f_2$};

			\draw[-] (23/3,1) -- (32/3,0);
			\draw[-] (23/3,1) -- (13/3,11);

			\draw[-] (1,23/3) -- (0,32/3);
			\draw[-] (1,23/3) -- (11,13/3);
			
			
			\draw[dashed] (23/3,1) -- (23/3,11);
			\draw[dashed] (23/3,1) -- (11,1);
			
			\draw[dashed] (1,23/3) -- (1,11);
			\draw[dashed] (1,23/3) -- (11,23/3);
			
			\fill (23/3,1) node[left]{$f(x_1)$} circle (3pt);
			\fill (1,23/3) node[below]{$f(x_2)$} circle (3pt);
			\fill (4.5,4.5) node[below left]{$f(x_3)$} circle (3pt);
			
			\draw (2,23/3) arc (0:-18.43:1);
			\node[] at (2.7,7.35)  {$\mtheta$};
			
			\draw (23/3+1,1) arc (0:-18.43:1);
			\node[] at (23/3+1.7,0.7)  {$\mtheta$};
			
			\draw (23/3+0.7,1) arc (0:-18.43:0.7);
			\draw (23/3+0.7,1) arc (0:108.43:0.7);
			\node[] at (23/3+0.3,1+0.3)  {$\gamma$};
			
			\draw (1.7,23/3) arc (0:-18.43:0.7);
			\draw (1.7,23/3) arc (0:108.43:0.7);
			\node[] at (1+0.3,23/3+0.3)  {$\gamma$};
			
			\draw[-] (1,0.1) -- (1,-0.1) node[below] {$\alpha \cdot \varepsilon'$};
			\draw[-] (4.5,0.1) -- (4.5,-0.1) node[below] {$1$};
			
			\draw[-] (0.1,1) -- (-0.1,1) node[left] {$\alpha \cdot \varepsilon'$};
			\draw[-] (0.1,4.5) -- (-0.1,4.5) node[left] {$1$};
			\end{tikzpicture}
			\caption{Illustration of the instance~$I$ constructed in the proof of Proposition~\ref{prop:tightness}. The shaded region is $\alpha$-approximated by $x_1$ or $x_2$  with respect to~$\leqq_\gamma^\utheta$ for $\utheta = \mtheta$. It is easy to see that for $\utheta\leq \mtheta$ (i.e., if the dominance cones are rotated counterclockwise in the picture), $x_3$~is $\alpha$-approximated by~$x_1$ with respect to~$\leqq_\gamma^\utheta$, and, for $\utheta\geq \mtheta$ (i.e., if the dominance cones are rotated clockwise), $x_3$~is $\alpha$-approximated by~$x_2$ with respect to~$\leqq_\gamma^\utheta$. Thus, $\{x_1,x_2\}$ is an $\alpha$-approximation with respect to~$\leqq_\gamma^\utheta$ for any $\utheta \in [0,\gamma - \frac \pi 2]\setminus \{\gamma - \pi, \frac \pi 2\}$. \label{fig:tightness}}
		\end{center}
	\end{figure}

	\section{Structural Results for Maximization Problems}\label{sec:maximization}
	In this section, we investigate whether the results obtained in Section~\ref{sec:minimization} can be transfered to the case of maximization.
	It is known that obtaining approximations using the weighted sum scalarization is more challenging for maximization problems than for minimization problems since the set of supported solutions does not yield any approximation guaranteein general~\cite{Bazgan+etal.:power-weighted-sum}. We will see that this is also the case when using general ordering cones to obtain approximations. In contrast to the case of minimization problems, where the approximation guarantee that is achieved by the set of $\gamma$-supported solutions increases continuously when $\gamma$ is increased between~$\frac \pi 2$ and~$\pi$, the set of $\gamma$-supported solutions does not yield any approximation guarantee for any~$\gamma > \frac \pi 2$.
	
	In this section, instead of the assumption that the set $f(X) + \R^p_\geqq$ is closed, we assume that $f(X) - \R^p_\geqq$ is closed and that $f(X)$ is bounded. The additional assumption of~$f(X)$ being bounded ensures external stability, i.e., that, also for maximization problem instances, for any feasible solution~$x \in X$ that is dominated by another feasible solution~$x' \in X$, there also exists an efficient solution~$x'' \in X_E$ dominating~$x$. 
	\textcolor{black}{All other underlying concepts in this section are analogous to the corresponding concepts for minimization problems introduced in Section~\ref{sec:preliminaries}.}

	\textcolor{black}{Observation~\ref{obs:subset-approx} transfers} directly to the case of maximization. However, results similar to Section~\ref{sec:minimization} do not hold for maximization. The set of $\gamma$-supported solutions does \emph{not} yield any approximation guarantee in general:
	
	\begin{theorem}\label{thm:maximization}
		For any~$\gamma \in (\frac \pi 2,\pi]$ and any~$\alpha \geq 1$, there exists an instance~$I$ of a biobjective maximization problem where the set of $\gamma$-supported solutions is not an $\alpha$-approximation.
	\end{theorem}
	\begin{proof}
		For $\gamma \in (\frac \pi 2, \pi]$ and $\alpha\geq 1$, define the following instance of a biobjective maximization problem (see also Figure~\ref{fig:maximization-example}):
		Let the feasible set consist of exactly three solutions, $x_1,x_2,x_3$ such that $f_1(x_1) = 1$, $f_2(x_1) =  \alpha + 2 + \frac 1 {\tan \mtheta} \cdot \alpha$, $f_1(x_2) = \alpha+2 + \frac 1 {\tan \mtheta} \cdot \alpha$, $f_2(x_2) = 1$, $f_1(x_3) = \alpha + 1$, and $f_2(x_3) = \alpha + 1$.
		Then $x_3$ is not $\gamma$-supported: If $\utheta \leq \mtheta$, we have $\tan \mtheta \leq \tan \otheta$ and, therefore,
		\begin{align*}
		\cos \otheta \cdot f_1(x_1) + \sin \otheta \cdot f_2(x_1) &= \cos \otheta + \sin \otheta \cdot \alpha + 2 \cdot \sin \otheta + \frac{\sin \otheta}{\tan \mtheta}\cdot \alpha\\
		&> \cos \otheta + \sin \otheta \cdot \alpha + \sin \otheta + \frac{\sin \otheta}{\tan \mtheta}\cdot \alpha\\
		& \geq \cos \otheta + \sin \otheta \cdot \alpha + \sin \otheta + \cos \otheta \cdot \alpha\\
		&= \cos \otheta \cdot f_1(x_3) + \sin \otheta \cdot f_2(x_3).
		\end{align*}
		Moreover, we have~$\tan \utheta \leq \frac 1{\tan \otheta} \leq \frac 1 {\tan \mtheta}$ by Lemma~\ref{lem:utheta-otheta-leq-mtheta}, which implies that
		\begin{align*}
		\sin \utheta \cdot f_1(x_1) + \cos \utheta \cdot f_2(x_1) &= \sin \utheta + \cos \utheta \cdot \alpha + 2 \cdot \cos \utheta + \frac{\cos \utheta}{\tan \mtheta}\cdot \alpha\\
		&> \sin \utheta + \cos \utheta \cdot \alpha + \cos \utheta + \frac{\cos \utheta}{\tan \mtheta}\cdot \alpha\\
		& \geq \sin \utheta + \cos \utheta \cdot \alpha + \cos \utheta + \sin \utheta \cdot \alpha\\
		&= \sin \utheta \cdot f_1(x_3) + \cos \utheta \cdot f_2(x_3).
		\end{align*}
		Thus, $x_3$ is dominated by $x_1$ in $I_\gamma^\utheta$. Similarly, if $\utheta \geq \mtheta$, the solution~$x_3$ is dominated by $x_2$ in $I_\gamma^\utheta$.
		On the other hand, $\{x_1,x_2\}$ is obviously not an $\alpha$-approximation.
		\qed
	\end{proof}
	\begin{figure}[ht!]
		\begin{center}
			\begin{tikzpicture}[scale=0.8]
			
			
			
			\draw[->] (-0.2,0) -- (7.5,0) node[below right] {$f_1$};
			\draw[->] (0,-0.2) -- (0,7.5) node[above left] {$f_2$};
			
			
			
			\draw[-] (0.5,6.5) -- (3.75,0);
			\draw[-] (0.5,6.5) -- (0,6.75);
			
			\draw[-] (6.5,0.5) -- (0,3.75);
			\draw[-] (6.5,0.5) -- (6.75,0);
			
			
			

			\fill (0.5,6.5) node[above right]{$f(x_1)$} circle (3pt);
			\fill (6.5,0.5) node[above right]{$f(x_2)$} circle (3pt);
			\fill (2,2) node[below]{$f(x_3)$} circle (3pt);
			

			\draw[-] (0.50,0.1) -- (0.50,-0.1) node[below] {$1$};
			\draw[-] (2,0.1) -- (2,-0.1) node[below] {$\alpha + 1$};
			\draw[-] (0.1,0.5) -- (-0.1,0.5) node[left] {$1$};
			\draw[-] (0.1,2) -- (-0.1,2) node[left] {$\alpha + 1$};
			\end{tikzpicture}
			\caption{Illustration of the maximization problem instance~$I$ constructed in the proof of Theorem~\ref{thm:maximization}. The dominance cones of $x_1$ and $x_2$ with respect to $\geqq_\gamma^\utheta$ are illustrated for $\utheta = \mtheta$. It is easy to see that $x_3$ is dominated by $x_1$ for $\utheta \leq \mtheta$ (if the dominance cones are rotated counterclockwise) and by $x_2$ for $\utheta \geq \mtheta$ (if the dominance cones are rotated clockwise). Thus, $x_3$ is not $\gamma$-supported. However, $x_3$ is not $\alpha$-approximated by $x_1$ or by $x_2$ in $I$.\label{fig:maximization-example}}
		\end{center}
	\end{figure}
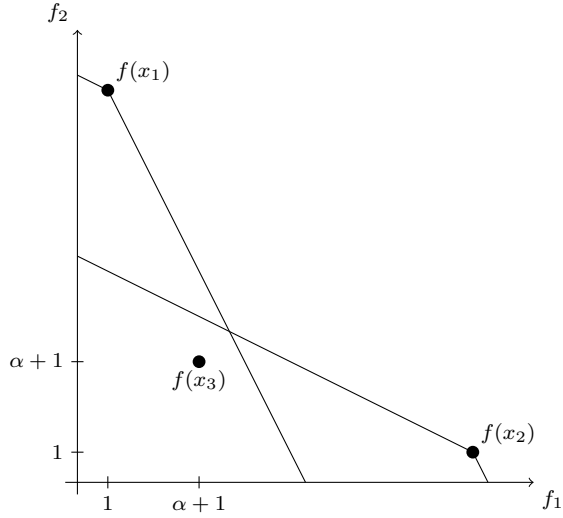
	
	\section{Conclusions and Additional Notes}
	This article studies approximation properties of general ordering cones containing the Pareto cone for biobjective minimization problems. As expected, it does not suffice to consider the set of optimal solutions (or an approximation) with respect to a single ordering cone in order to achieve an approximation guarantee in the classical sense. Instead, we classify ordering cones by their inner angle~$\gamma$ and consider sets that are optimal (or approximate) with respect to all closed convex ordering cones of inner angle~$\gamma$ simultaneously. These sets then, in fact, achieve an approximation guarantee, which depends on~$\gamma$. We introduce the concept of $\gamma$-supportedness to describe solutions that are optimal with respect to at least one ordering cone of inner angle~$\gamma$. Since this concept incorporates both efficiency and supportedness as special cases, our results are a generalization of the fact that the efficient set is a $1$-approximation and of known results about the approximation quality achievable by the set of supported solutions. Our results are best possible in the sense that better approximation guarantees than the ones shown are not generally achievable for any inner angle $\gamma \in [\frac \pi 2, \pi]$.
	
	\medskip
	
	Designing (polynomial-time) approximation algorithms based on general ordering cones (other than weighted sum scalarizations) is possible but presents further challenges since the resulting problems stay biobjective. Moreover, when attempting to compute, e.g., $\gamma$-supported solutions via the definition of $\gamma$-supportedness, all values for $\utheta$ from the continuous set $[0,\gamma -\frac \pi 2] \setminus \{\gamma - \pi, \frac\pi 2\}$ have to be considered. Finally, the fact that the matrix describing the linear mapping $T_\gamma^\utheta$ typically contains irrational entries constitutes an additional obstacle for algorithmic applications of the presented concepts.
	
	\medskip
	
	An interesting direction for future research is the generalization of the presented results to general ordering cones in more than two objectives. The equivalence between closed convex cones containing $\R_\geqq$ and closed vector preorders satisfying the Pareto axiom also holds for the more general case of three or more objectives. Also, most of the definitions and observations stated in Section~\ref{sec:preliminaries} easily transfer to the case of more than two objectives. For details, we refer to~\cite{Vanderpooten+etal:covers+approximations}. 
	
	Moreover, Proposition~\ref{prop:onecone} can easily be generalized to $p \geq 3$ objectives: For any closed vector preorder $\preceq$ on $\R^p$ satisfying the Pareto axiom (except for~$\leqq$) and any $\alpha \geq 1$, there exists a $p$-objective minimization problem instance where the set of optimal solutions with respect to $\preceq$ is not an $\alpha$-approximation with respect to~$\leqq$.
	
	However, since, in three or more dimensions, a general closed convex cone cannot be described by a finite number of scalar parameters, generalizing the positive results from Section~\ref{sec:minimization} is far from straightforward.
	One way to simplify the situation is the restriction to \emph{polyhedral cones}, which are cones that can be obtained from the nonnegative orthant via a linear mapping. Nevertheless, even then, it is not obvious how to generalize the concept of $\gamma$-supportedness, as there does not exist an unambiguous inner angle in a polyhedral cone in three or more dimensions.
	
	\section*{Declarations}
	The authors have no conflicts of interest to declare that are relevant to the content of this article. Data sharing not applicable to this article as no datasets were generated or analysed during the current study. 
	
	\begin{acknowledgements}
		This work was supported by the DFG grants RU~1524/6-1 and TH~1852/4-1.
	\end{acknowledgements}
	
	\bibliographystyle{spmpsci} 
	\bibliography{Literatur}   
	
	\section*{Appendix A - Proof of Lemma~\ref{lem:utheta-otheta-leq-mtheta}}
	\begin{proof}
		If $\utheta = 0$ or $\otheta = 0$, the claim trivially holds, so we assume that $\gamma > \frac \pi 2$ and $\utheta \in (0, \gamma- \frac \pi 2)$. First, note that
		\begin{align}\label{eq:trigonometry}
		\tan \otheta = \tan (\gamma- \frac \pi 2 - \utheta) = - \frac{1}{\tan(\gamma - \utheta)} = - \frac{1+\tan\gamma \tan \utheta}{\tan \gamma - \tan \utheta},
		\end{align}
		where the second equality follows from the symmetry of the $\tan$-function and the last equality follows from the addition formula for $\tan$, which states that, for any $\theta_1,\theta_2 \in \R$ for which $\tan \theta_1$, $\tan \theta_2$, and $\tan(\theta_1-\theta_2)$ are well-defined,
		\begin{align*}	
		\tan(\theta_1 - \theta_2) &= \frac{\tan \theta_1 - \tan \theta_2}{1 + \tan \theta_1 \cdot \tan \theta_2}.
		\end{align*}
		Define $s \coloneqq \frac 1 2 \cdot \left(\frac{\sqrt{\tan \utheta}}{\sqrt{\tan \otheta}} + \frac{\sqrt{\tan \otheta}}{\sqrt{\tan \utheta}}\right)$. Then $s \geq 1$, where $s = 1$ if and only if $\utheta = \otheta = \mtheta$. Moreover, we can write
		\begin{align*}
		\tan \utheta + \tan \otheta &= \sqrt{\tan \utheta} \cdot \sqrt{\tan \otheta} \cdot \frac{\tan \utheta + \tan \otheta}{ \sqrt{\tan \utheta} \cdot \sqrt{\tan \otheta}}\\
		&= \sqrt{\tan \utheta} \cdot \sqrt{\tan \otheta} \cdot \left(\frac{\sqrt{\tan \utheta}}{\sqrt{\tan \otheta}} + \frac{\sqrt{\tan \otheta}}{\sqrt{\tan \utheta}}\right)\\
		&= 2s \cdot \sqrt{\tan \utheta} \cdot \sqrt{\tan \otheta}.
		\end{align*}
		Now, we reformulate \eqref{eq:trigonometry} to obtain
		\begin{align*}
		\tan \utheta \cdot \tan \otheta - \tan \utheta \cdot \tan \gamma - \tan \otheta \cdot \tan \gamma = 1.
		\end{align*}
		This yields
		\begin{align*}
		1+ s^2 \cdot (\tan \gamma)^2 &= \tan \utheta \cdot \tan \otheta - (\tan \utheta + \tan \otheta) \cdot \tan \gamma + s^2 \cdot (\tan \gamma)^2\\
		&= \tan \utheta \cdot \tan \otheta - 2s \cdot \tan \gamma \cdot \sqrt{\tan \utheta} \cdot \sqrt{\tan \otheta} + s^2 \cdot (\tan \gamma)^2\\
		&= \left(\sqrt{\tan \utheta} \cdot \sqrt{\tan \otheta} - s \cdot \tan \gamma\right)^2
		\end{align*}
		and, thus,
		\begin{align}\label{eq:right-hand-side}
		\sqrt{\tan \utheta} \cdot \sqrt{\tan \otheta} = \sqrt{1+ s^2 \cdot (\tan \gamma)^2} + s \cdot \tan \gamma.
		\end{align}
		By plugging the case that $\utheta = \otheta = \mtheta$ into~\eqref{eq:right-hand-side} and using that $s \geq 1$, we obtain
		\begin{align*}
		\tan \mtheta = \sqrt{\tan \mtheta} \cdot \sqrt{\tan \mtheta} &= \sqrt{1+ (\tan \gamma)^2} + \tan \gamma\\
		&\geq \sqrt{1+ s^2 \cdot (\tan \gamma)^2} + s \cdot \tan \gamma\\
		&= \sqrt{\tan \utheta} \cdot \sqrt{\tan \otheta},
		\end{align*}
		where the inequality holds since $\tan \gamma \leq 0$ and, therefore, the right hand side of~\eqref{eq:right-hand-side} is non-increasing in $s$. \qed
	\end{proof}
	
	\section*{Appendix B - Proof of Corollary~\ref{cor:alternative-formulas}}
	\begin{proof}
		\enlargethispage{\baselineskip}
		By Theorem~\ref{thm:cone-approx}, we know that $X_Q$ is an $(\alpha \cdot (1+\tan \mtheta))$-approximation, where $\tan \mtheta  = \tan \left(\frac \gamma 2 - \frac \pi 4\right) = \tan\left(\frac{\gamma - \frac \pi 2}{2}\right)$.
		The well-known half-angle formula for $\tan$ states that, for any angle $\theta \in [0,\pi)$,
		\begin{align*}
		\tan \frac \theta 2= \frac{\sin \theta}{1 + \cos \theta} = \frac{1- \cos  \theta}{\sin \theta}.
		\end{align*}
		Thus, on the one hand, we can write $S$ as
		\begin{align*}
		S = \tan\left(\frac{\gamma - \frac \pi 2}{2}\right) = \frac{\sin\left(\gamma - \frac \pi 2\right)}{\cos\left(\gamma - \frac \pi 2\right) + 1} = \frac{- \cos \gamma}{\sin \gamma + 1}
		\end{align*}
		and, on the other hand, we can write $S$ as
		\begin{align*}
		S =\tan\left(\frac{\gamma - \frac \pi 2}{2}\right) &= \frac{1-\cos\left(\gamma - \frac \pi 2\right)}{\sin\left(\gamma - \frac \pi 2\right)} = \frac{1 -\sin \gamma}{-\cos \gamma} = \frac{1}{-\cos \gamma} + \tan \gamma\\ &= \sqrt{1+(\tan \gamma)^2} + \tan \gamma,
		\end{align*}
		where the last equality follows from the well-known identity
		\begin{align*}
		\cos \theta &=  -\frac{1}{\sqrt{1+ (\tan \theta)^2}}
		\end{align*}
		for $\theta \in (\frac \pi 2, \pi]$.
		\qed
	\end{proof}
	
\end{document}